\documentclass[11pt]{amsart}
\usepackage{amsfonts, latexsym, amssymb, amsgen, amsbsy, amstext, amsopn, amsmath, txfonts, mathrsfs, yfonts}
\usepackage{color}
\usepackage{verbatim}
\usepackage[toc,page]{appendix}
\usepackage{hyperref}

\newcommand{\bbN}{\mathbb{N}}
\newcommand{\Q}{\mathbb{Q}}
\newcommand{\bbR}{\mathbb{R}}

\newcommand{\N}{\mathbb{N}}

\newcommand{\vG}{\varmathbb{G}}

\newcommand{\vR}{\varmathbb{R}}
\newcommand{\vM}{\varmathbb{M}}

\newcommand{\vW}{\varmathbb{W}}

\newcommand{\vX}{\varmathbb{X}}
\newcommand{\vY}{\varmathbb{Y}}

\newcommand{\vV}{\varmathbb{V}}

\newcommand{\cI}{\mathcal{I}}
\newcommand{\cL}{\mathcal{L}}
\newcommand{\cN}{\mathcal{N}}

\newcommand{\cJ}{\mathcal{J}}

\newcommand{\cS}{\mathcal{S}}

\newcommand{\cp}{\mathfrak{p}}
\newcommand{\cv}{\mathfrak{v}}

\newcommand{\cw}{\mathfrak{w}}

\newcommand{\dist}{\mbox{dist}}

\newcommand{\pa}[1]{\left( #1 \right)}               
\newcommand{\set}[1]{\left\{ #1 \right\}}            
\newcommand{\pal}[1]{\left| #1 \right|}            

\newcommand{\ep}{\varepsilon}

\newcommand{\ph}{\varphi}

\newcommand{\sm}{\setminus}

\newcommand{\lan}{\langle}
\newcommand{\ran}{\rangle}

\newcommand{\lra}{\longrightarrow}

\newcommand{\der}{\partial}

\newtheorem{Teo}{Theorem}
\newtheorem{Claim}{Claim}

\newtheorem{Lem}[Teo]{Lemma}
\newtheorem{Pro}[Teo]{Proposition}

\theoremstyle{definition}
\newtheorem{Def}[Teo]{Definition}

\newtheorem{Rem}[Teo]{Remark}

\numberwithin{Teo}{section}
\numberwithin{equation}{section}

\begin{document}
\title[Porosity and Differentiability]
{
Porosity and differentiability of Lipschitz maps from stratified groups to Banach homogeneous groups}
\author[Valentino Magnani]{Valentino Magnani}
\address[Valentino Magnani]{Department of Mathematics, University of Pisa, Largo Pontecorvo 5, 56127 Pisa, Italy}
\email{valentino.magnani@unipi.it}
\author[Andrea Pinamonti]{Andrea Pinamonti}
\address[Andrea Pinamonti]{Department of Mathematics, University of Trento, Via Sommarive 14, 38050 Povo (Trento), Italy}
\email{Andrea.Pinamonti@gmail.com}
\author[Gareth Speight]{Gareth Speight}
\address[Gareth Speight]{Department of Mathematical Sciences, University of Cincinnati, 2815 Commons Way, Cincinnati, OH 45221, United States}
\email{Gareth.Speight@uc.edu}

\date{\today}
\thanks{V.~M.~acknowledges the support of the University of Pisa, (Institutional Research Grant) Project\ PRA\_2016\_41. A.~P.~acknowledges the support of the Istituto Nazionale di Alta  Matematica F. Severi.}
\renewcommand{\subjclassname} {\textup{2010} Mathematics Subject Classification}



\begin{abstract}
Let $f$ be a Lipschitz map from a subset $A$ of a stratified group to a Banach homogeneous group. We show that directional derivatives of $f$ act as homogeneous homomorphisms at density points of $A$ outside a $\sigma$-porous set.
At all density points of $A$ we establish a pointwise characterization of differentiability in terms of  directional derivatives. These results naturally lead us to an alternate proof of almost everywhere differentiability of Lipschitz maps from subsets of stratified groups to Banach homogeneous groups satisfying a suitably weakened Radon-Nikodym property.
\end{abstract}

\maketitle

\section{Introduction}\label{intro}

Stratified groups are a special class of finite dimensional, connected, simply connected and nilpotent Lie groups (Definition~\ref{Def:stratified}). These groups, equipped with the so-called {\em homogeneous norm}, were introduced by Folland in the framework of subelliptic PDE \cite{Fol75}. They subsequently appeared in the work of Pansu under the name of Carnot groups, where their metric structure was defined by the so-called Carnot-Carath\'eodory distance \cite{PP}. In the last two decades there has been increasing interest in the relationship between the geometry of stratified groups and other areas of mathematics, such as PDE, differential geometry, control theory, geometric measure theory, mathematical finance and robotics \cite{ABB, B94, BLU, CDPT07, Gro96, Mon02, SC16}. 

Nevertheless several interesting and challenging problems remain open, since the geometry of stratified groups is highly non-trivial. For instance, rectifiability of the reduced boundary in higher step groups, the isoperimetric problem, the regularity of minimal surfaces, the regularity of geodesics and the validity of general coarea formulae for Lipschitz mappings.
One of the basic reasons for these difficulties is that any noncommutative stratified group contains no subset of positive measure that is bi-Lipschitz equivalent to a subset of a Euclidean space \cite{AmbKir00, HajMal15,Mag04, Sem96}. In all of these results 
the key role is played by the Pansu's generalization of Rademacher's theorem to Carnot groups \cite{PP}.

This differentiation theorem states that every Lipschitz map from an open subset of a stratified group to another stratified group is differentiable almost everywhere with respect to the natural Haar measure. Here differentiability is defined like classical differentiability, but takes into account the new geometric and algebraic structure (Definition~\ref{def:hdiff}). More broadly, a great effort has been made to understand differentiability properties of Lipschitz functions in different settings, like stratified groups, Banach spaces and more general metric measure spaces. We mention 
only a few papers \cite{Alb,Bate15,Che99,CK, CMPSC1,CMPSC2,FSC2,LP03,Mag,PinS15,Pre90,PS15,PZ01} to give a glimpse
of the many works in this area.

In the setting of noncommutative stratified groups, the first named author and Rajala extended Pansu's theorem to Lipschitz mappings on a measurable set of a stratified group and taking values in a \emph{Banach homogeneous group}
(Definition~\ref{bhg}). They proved the following theorem \cite{MagTap14}.

\begin{Teo}\label{thm:Pansudiff}
Suppose $\vM$ is a Banach homogeneous group whose first layer 
$H_1 \subset \vM$ has the RNP. If $f \colon A \to \vM$ is Lipschitz,
$A \subset \vG$ and $\vG$ is a stratified group, then $f$ 
is almost everywhere differentiable.
\end{Teo} 

It is easy to notice that every stratified group is a Banach homogeneous group and the class of finite dimensional Banach homogeneous groups coincides with that of homogeneous groups, as defined in \cite{FS82}.
Commutative Banach homogeneous groups coincide with Banach spaces. The simplest example besides these is an infinite dimensional version of the Heisenberg group $H^2\times\vR$. We consider the infinite dimensional real Hilbert space $H$ with scalar product $\left\langle\cdot ,\cdot \right\rangle$ and the group operation is defined as follows:
\[
(h_1,h_2,t_1)\cdot (h_1',h_2',t_1')=(h_1+h_1', h_2+h_2', t_1+t_2+\langle h_1, h_2' \rangle- \langle h_2, h_1' \rangle),
\]
for any $(h_1,h_2,t_1), (h_1',h_2',t_1')\in H^{2} \times\vR$. We refer the interested reader to \cite{MagTap14} for more examples.

The present paper continues the investigation started in \cite{PinS16}, on
the relationship between porosity and differentiability of Lipschitz maps on stratified groups. A set in a metric space is (upper) porous (Definition~\ref{def_porous}) if each of its points sees nearby relatively large holes in the set on arbitrarily small scales. A set is $\sigma$-porous if it is a countable union of porous sets. Proving a set is $\sigma$-porous is useful because $\sigma$-porous sets in metric spaces are first category and have measure zero with respect to any doubling measure \cite{Zaj87, Zaj05}. Showing a set is $\sigma$-porous usually gives a stronger result than showing it is of first category or has measure zero.

It is well known that porosity is a useful tool to study differentiability of Lipschitz mappings. We refer the reader to the survey articles \cite{Zaj87, Zaj05} for more information about porous sets and their applications. Porosity has been also recently used in Euclidean spaces to obtain the following result: for any $n>1$, there exists a Lebesgue null set in $\mathbb{R}^{n}$ containing a point of differentiability for every Lipschitz mapping from $\mathbb{R}^{n}$ to $\mathbb{R}^{n-1}$ \cite{PS15}. In the Banach space setting, \cite{LP03, LPT12} give a version of Rademacher's theorem for Frech\'{e}t differentiability of Lipschitz mappings on Banach spaces in which porous sets are negligible in a suitable sense. Other results have also been studied in stratified groups \cite{LPS17, PinS15, PS3, PS4}. 

In \cite{PinS16}, the second and third named authors used porosity to study some fine differentiability properties of Lipschitz mappings from stratified groups to Euclidean spaces. 
In the present paper, we improve upon \cite{PinS16} in two directions: we consider more general classes of domains and targets of the Lipschitz mapping.
We allow the mapping to be defined on a measurable subset of a stratified group. Due to the absence in general of Lipschitz extensions between stratified groups, some technical difficulties appear. The Euclidean space in the target is replaced by a Banach homogeneous group. Thus, the noncommutative infinite dimensional structure of a Banach homogeneous group adds further difficulties. 

Let $A$ be a subset of a stratified group $\vG$,  $\vM$ be a Banach homogeneous group, and consider a Lipschitz map $f:A\subset\vG\to\vM$.
Our first result states that directional derivatives of $f$ at density points of the domain act as homogeneous homomorphisms outside a $\sigma$-porous set.  Here it is worth to emphasize that the notion of directional derivative (Definition~\ref{directionalderivative}) takes into account the fact that $f$ need not be defined on a whole neighborhood of the point. 
For the definition of homogeneous homomorphism, see Definition~\ref{h-hom}. Note that $\partial^{+}f(x, \zeta)$ and $\delta_{a}\zeta$ denote directional derivatives and dilations, as defined in Section \ref{preliminaries} and $D(A)$ denotes the set of density points of $A\subset \vG$.

\begin{Teo}\label{generalcase}
Suppose $A\subset \vG$ is measurable and $f\colon A\to \vM$ is Lipschitz. 
Then there is a $\sigma$-porous set $P\subset \vG$ such that directional derivatives act as homogeneous homomorphisms outside $P$. Namely, whenever $x\in A\cap D(A) \setminus P$ the following implication holds: if $\partial^{+}f(x, \zeta)$ and $\partial^{+}f(x, \eta)$ exist for some $\zeta,\eta\in \vG$, then $\partial^{+}f(x,\delta_{a}\zeta \delta_{b} \eta)$ exists for every $a,b>0$ and
\[
\partial^{+}f(x,\delta_{a}\zeta  \delta_{b} \eta)= \delta_{a}\partial^{+}f(x,\zeta) \delta_{b}\partial^{+}f(x,\eta).
\]
\end{Teo}
Surprisingly, the techniques of \cite[Theorem 2]{PZ01} can be extended to the nonlinear framework of stratified groups. Nontrivial obstacles arise from the structure of both the domain and the groups involved.
Theorem~\ref{generalcase} precisely generalizes \cite[Theorem 3.7]{PinS16} from real valued Lipschitz mappings to Banach homogeneous group valued Lipschitz mappings. 

For real valued Lipschitz mappings, directional derivatives were defined using the Lie algebra of the stratified group and only horizontal directions were used. 
In the present paper we realize that directional derivatives must be introduced for all directions, using the Lie group structures of $\vG$ and $\vM$. This important difference is necessary to reflect the noncommutative structure of the target and it was rather unexpected.
We notice, indeed, that the restriction of a Lipschitz mapping to a nonhorizontal curve is no longer Lipschitz in the Euclidean sense.

Our second main result characterizes points of differentiability at density points of the domain. For the definition of directional derivatives, see Definition~\ref{directionalderivative}.

\begin{Teo}\label{precisesubset}
Suppose $A\subset\vG$ is measurable, $f:A\to \vM$ is Lipschitz and $x\in A\cap D(A)$. 
Then $f$ is differentiable at $x$ if and only if the following properties hold: 
\begin{enumerate}
	\item $f$ is differentiable at $x$ in any direction $\zeta\in\vG\setminus\set{0}$ 
	\item the mapping $\vG\ni \zeta\to \partial^+f(x,\zeta)\in\vM$ is a homogeneous homomorphism.
\end{enumerate}
\end{Teo}

As a consequence of Theorem~\ref{generalcase} and Theorem~\ref{precisesubset} we prove that, at density points of the domain except for a $\sigma$-porous set, existence of directional derivatives in a spanning set of horizontal directions implies differentiability. If the horizontal space of $\vM$ satisfies the RNP then directional derivatives in this finite set of directions exist at almost every point of $A$ (Theorem~\ref{lineardensityderivatives}). As a byproduct, taking into account that porous sets have measure zero, we are lead to a different and simpler way to establish both Theorem~ \ref{thm:Pansudiff} and then also Pansu's Theorem \cite{PP}.
 
The results of this paper illustrate the role played by porosity to pass from almost everywhere existence of horizontal directional derivatives \cite[Theorem 3.1]{MagTap14} to almost everywhere differentiability of Lipschitz maps. The use of porosity provides a more natural approach to differentiability and provides more precise information about how differentiability occurs.

We now describe the structure of the paper. In Section~\ref{preliminaries} we give the main definitions. In Section~\ref{subsets} we prove Theorem~\ref{generalcase}, while in Section~\ref{diff} we prove Theorem~\ref{precisesubset}. Section~\ref{ProofMag} is devoted to a few technical facts. The first one is the construction of a separable Banach homogeneous target for our given Lipschitz mapping. We then give a proof of a technical lemma that we use earlier for estimating distances. Finally we prove that the set of points where $f$ is differentiable in some directions is measurable. This measurability is important to establish Theorem~\ref{lineardensityderivatives} and its proof is rather technical since our Lipschitz mapping cannot be extended to the whole space.

\vspace{.3cm}

\noindent \textbf{Acknowledgement.} The second and third named authors thank the organizers of the workshop ``Singular Phenomena and Singular Geometries", Pisa, 20-23 June 2016, for the warm hospitality and for the scientific environment that inspired the project of the present paper.

\section{Preliminaries}\label{preliminaries}

This section describes the main notions used throughout the paper.

\subsection{Banach Lie algebras and Banach homogeneous groups}

We recall only basic facts about Banach homogeneous groups. More information and additional examples can be found in \cite{MagTap14}.

\begin{Def}
A {\em Banach Lie algebra} is a Banach space $\vM$ equipped with a continuous Lie bracket, namely a continuous, bilinear and skew-symmetric mapping $[\cdot,\cdot]:\vM\times \vM\lra\vM$ that satisfies the Jacobi identity:
\[
[x,[y,z]]+[y,[z,x]]+[z,[x,y]]=0\qquad \mbox{for all }x,y,z\in\vM.
\]
A Banach Lie algebra $\vM$ is {\em nilpotent} if there is $\nu\in\bbN$ such that whenever $x_1,x_2,\ldots,x_{\nu+1}\in\vM$, we have
\[
[[ [\cdots[[x_1,x_2],x_3]\cdots], x_\nu],x_{\nu+1}]=0\quad
\]
and there exist $y_1,y_2,\ldots,y_\nu\in\vM$ such that 
\[
[[\cdots[[y_1,y_2],y_3]\cdots],y_\nu]\neq0.
\]
The integer $\nu$ is uniquely defined and is called the {\em step of nilpotence} of $\vM$.
\end{Def}

A nilpotent Banach Lie algebra $\vM$ can be equipped with the Lie group operation
\begin{equation}\label{gopeB}
xy=x+y+\sum_{m=2}^\nu P_m(x,y)\,,
\end{equation}
which is the truncated Baker-Campbell-Hausdorff series where the Banach Lie bracket is used. 
For any $m\geq 2$, the polynomial $P_m$ above is given by {\em Dynkin's formula}
\begin{equation}\label{dynkin}
P_m(x,y)= \mbox{\scriptsize $\displaystyle \sum \frac{(-1)^{k-1}}{k}\frac{m^{-1}}{p_1!q_1!\cdots p_k!q_k!}$}  \; 
\mbox{\small $\underbrace{x\circ\cdots\circ x}_{p_1\; times}\circ\overbrace{y\circ\cdots\circ y}^{q_1\; times}\circ \cdots\circ  \underbrace{x\circ\cdots\circ x}_{p_k\; times}\circ\overbrace{y\circ\cdots\circ y}^{q_k\; times }$ }.
\end{equation}
Here we used the nonassociative product
\[
x_{i_1}\circ x_{i_2}\circ\cdots \circ x_{i_k}=[[\cdots[[x_{i_1},x_{i_2}],x_{i_3}]\cdots],x_{i_k}]
\]
and the sum is taken over the $2k$-tuples $(p_1,q_1,p_2,q_2,\ldots,p_k,q_k)$ such that $p_i+q_i\geq 1$ for all positive $i,k\in\bbN$ and $\sum_{i=1}^kp_i+q_i=m$. Note that $P_{2}$ has the simple form $P_2(x,y)=[x,y]/2$. The explicit formula for the group product will only be used in the technical Section \ref{ProofMag}.

\begin{Def}
A nilpotent Banach Lie algebra equipped with the group operation \eqref{gopeB} is called a {\em Banach nilpotent Lie group}.
\end{Def}

If $S_1,S_2,\ldots,S_n\subset\vX$ are closed subspaces of a Banach space $\vX$ such that the mapping
\[
J\colon S_1\times\cdots\times S_n\lra\vX \quad\text{ with  } \quad 
J(s_1,\ldots,s_n)=\sum_{l=1}^n s_l
\]
is a Banach isomorphism (continuous linear bijection), then we write $\vX=S_1\oplus\cdots\oplus S_n$.

\begin{Def}\label{bhg}
We say that a Banach nilpotent Lie group $\vM$ of step $\nu$ is a {\em Banach homogeneous group} if it admits a stratification. This means there exist $\nu$ closed subspaces $H_1,\ldots,H_\nu$ such that
\[\vM=H_1\oplus\cdots\oplus H_\nu,\]
where $[x,y]\in H_{i+j}$ if $x\in H_i$ and $y\in H_j$ and $i+j\leq\nu$, and $[x,y]=0$ otherwise.
\end{Def}

The stratification equips any Banach homogeneous group $\vM$ with {\em dilations} $\delta_r\colon\vM\lra\vM$, $r>0$. These are Banach isomorphisms defined by
\[
\delta_r\circ\pi_i=r^i\,\pi_i
\]
where $\pi_i:\vM\to H_i$ is the canonical projection for all $i=1,\ldots,\nu$. These dilations respect the Lie bracket and Lie group structure:
\[
[\delta_rx,\delta_ry]=\delta_r[x,y]\quad\text{ and } \quad \delta_{r}(xy)=\delta_{r}(x)\delta_{r}(y) \quad \text{ for all }x,y\in \vM, \ r>0.
\]

Throughout this article we denote by $|\cdot|$ the underlying Banach space norm on $\vM$. It can be shown there exist positive constants $\sigma_1,\ldots,\sigma_\nu$ with $\sigma_1=1$ such that $\|\cdot \|\colon \vM\to \vR$ defined by:
\begin{equation}\label{Banhom}
\|x\,\|=\max\{\sigma_i|\pi_i(x)|^{1/i}:\,1\leq i\leq\nu\}
\end{equation}
satisfies the standard properties
\[\|\delta_rx\|=r\,\|x\|\quad \mbox{ and }\quad \|xy\|\leq\|x\|+\|y\| \quad \mbox{ for all }x, y\in \vM, \ r>0.         \]
We say that $\|\cdot\|$ is a {\em Banach homogeneous norm} on $\vM$. The map $\rho(x,y)=\|x^{-1}y\|$ is a distance on $\vM$ satisfying
\begin{equation}\label{eq:dist_hom}
\rho(zx,zy)=\rho(x,y)\quad \mbox{ and }\quad \rho(\delta_rx,\delta_ry)=r\,\rho(x,y) \quad \text{ for all }x,y,z\in\vM, \ r>0.
\end{equation}
We say that $\rho$ is a {\em Banach homogeneous distance} on $\vM$. We also write $\rho(z)=\rho(z,0)$ for $z\in\vM$.

An important subclass of Banach homogeneous groups are stratified groups.

\begin{Def}\label{Def:stratified}
A \emph{stratified group} $\vG$ is a Banach homogeneous group for which the underlying Banach space is finite dimensional and whose first layer $V_1$ of the stratification $\vG=V_1\oplus\cdots\oplus V_s$ satisfies 
\[
[V_1,V_j]=V_{j+1} \text{ for }j\geq 1 \qquad \text{and} \qquad V_j=\set{0} \text{ for }j>s.
\]
As a finite dimensional Lie group, $\vG$ is equipped with a Haar measure which we denote by $\mu$.
\end{Def}
\begin{Rem}
While any stratified group is also a finite dimensional Banach homogeneous group, the converse does not hold. To see this it suffices to consider a commutative Banach homogeneous group $\vM=H_1\oplus H_2$, where  $H_1$ and $H_2$ are finite dimensional vector spaces. Clearly the commutative Lie product yields $[H_1,H_1]=\set{0}\subset H_2$, so $\vM$ cannot be a stratified group.
\end{Rem}
We can equip any stratified group $\vG$ with a {\em homogeneous distance} (i.e. a distance satisfying \eqref{eq:dist_hom}). The Haar measure $\mu$ then satisfies
\[
\mu(xA)=\mu(A) \quad \mbox{and}\quad \mu(\delta_{r}(A))=r^{Q}\mu(A) \quad \text{for measurable }A\subset \vG \text{ and }x\in \vG
\]
and
\[\mu(B(x,r))=r^{Q}\mu(B(0,1))\quad \text{ for every open ball }B(x,r)\subset \vG.\]

Fixing a basis compatible with the stratification identifies $\vG$ with $\bbR^n$ for a positive integer $n$. In these coordinates, the Haar measure is simply the $n$-dimensional Lebesgue measure $\mathcal{L}^n$ (up to constant multiplication).

From now on $\vM$ will be a Banach homogeneous group with stratification $\vM=H_1\oplus\cdots\oplus H_\nu$ equipped with the Banach homogeneous norm $\|\cdot \|$ and corresponding homogeneous distance $\rho$. We denote by $\vG$ a stratified group of step $s$ with stratification $\vG=V_1\oplus\cdots\oplus V_s$ equipped with a homogeneous distance $d$. We also use the notation $d(x)=d(x,0)$.

\subsection{Directional derivatives and differentiability}\label{sub:Direc}

For functions defined on general measurable domains we may not be able to approach each point from every direction inside the domain. We will introduce an unusual but flexible definition of directional derivative, based on the following approximation of directions inside a domain.

\begin{Def}\label{approximatedirection}
Let $A\subset \vG$ be measurable and $x\in \vG$. We say that points $\zeta_{x}^{t}\in \vG$ for $t>0$ \emph{approximate in $A$ the direction $\zeta$ at the point $x$} if
\begin{equation}\label{lim}
x\zeta_{x}^{t}\in A \mbox{ for every }t>0 \qquad \mbox{and} \qquad \frac{d(\zeta_{x}^{t}, \delta_{t}\zeta)}{t}\to 0 \quad \mbox{as}\quad t\downarrow 0.
\end{equation}
\end{Def}

\begin{Def}
Let $A\subset \vG$ be a measurable set. We say that $x\in \vG$ is a \emph{density point} of $A$ if
\[\lim_{r\downarrow 0}\frac{\mu(A\cap B(x,r))}{\mu(B(x,r))}=1.\]
We denote by $D(A)$ the set of density points of $A$.
\end{Def}

Since $\mu$ is doubling, the Lebesgue differentiation theorem applies and $\mu(A\setminus D(A))=0$ for any measurable set $A\subset \vG$. It is a standard fact that at every density point $x\in D(A)$ we have
\begin{equation}\label{eq:distAdensity}
\frac{\dist(A,y)}{d(x,y)}\to0\quad \text{as $y\to x$}.
\end{equation}

At density points of a measurable set, we can approximate in every direction.

\begin{Lem}\label{convzeta}
Suppose $A\subset \vG$ is measurable and $x\in D(A)$. Then for every $\zeta \in \vG$ there exist approximating points $\zeta_{x}^{t}$ in direction $\zeta$ at the point $x$. In addition, we can choose $\zeta_{x}^{t}$ such that the limit in \eqref{lim} is uniform for $0<d(\zeta)\leq 1$.
\end{Lem}

\begin{proof}
For each $t>0$, choose $y_t\in A$ such that 
\[
d(y_t,x\delta_t\zeta)<\dist(A,x\delta_t\zeta)+t^2 d(\zeta).
\]
To conclude the proof, apply \eqref{eq:distAdensity} with $y=x\delta_t\zeta$ and define $\zeta_x^t=x^{-1}y_t$.
\end{proof}

\begin{Def}\label{directionalderivative}
Let $A\subset \vG$ be measurable and $f\colon A\to \vM$. Suppose $x\in A$ and $\zeta\in \vG\setminus\set{0}$ for which there exist points which approximate in $A$ the direction $\zeta$ at $x$.

We say that $f$ is \emph{differentiable at $x$ in direction $\zeta$} with directional derivative $\partial^{+}f(x,\zeta)\in \vM$ if for any choice $t\mapsto \zeta_{x}^{t}$ of points approximating in $A$ the direction $\zeta$ at $x$, we have
\[ \lim_{t\downarrow 0} \delta_{1/t} \Big( f(x)^{-1}f(x\zeta_{x}^{t}) \Big) =  \partial^{+}f(x,\zeta).\]
\end{Def}

\begin{Pro}\label{dilatedirectional}
Assume the hypotheses of Definition~\ref{directionalderivative}. Then the direction $\delta_{a}\zeta$ can be approximated at $x$ for every $a>0$. Further, if $\partial^{+}f(x,\zeta)$ exists then also $\partial^{+}f(x,\delta_a\zeta)$ exists for every $a>0$ and we have $\partial^{+}f(x,\delta_a\zeta)=\delta_a \partial^{+}f(x,\zeta)$.
\end{Pro}

\begin{proof}
Let $\zeta_{x}^{t}$ approximate the direction $\zeta$ at $x$. Then $x\zeta_{x}^{at}\in A$ and
\[\frac{d(\zeta_{x}^{at},\delta_{t}\delta_{a}\zeta)}{t}=\frac{ad(\zeta_{x}^{at},\delta_{at}\zeta)}{at}\to 0.\]
Hence $\zeta_{x}^{at}$ approximates the direction $\delta_{a}\zeta$ at $x$. 

Now suppose $\eta_{x}^{t}$ approximates the direction $\delta_{a}\zeta$ at $x$. By the above argument, $\eta_{x}^{t/a}$ approximates the direction $\zeta$ at $x$. Using differentiability in direction $\zeta$ gives
\[\lim_{t\downarrow 0} \delta_{1/t}(f(x)^{-1}f(x\eta_{x}^{t}))=\delta_{a}\lim_{t\downarrow 0} \delta_{1/t}(f(x)^{-1}f(x\eta_{x}^{t/a}))=\delta_{a}\partial^{+}f(x,\zeta).\]
Hence $ \partial^{+}f(x,\delta_{a}\zeta)=\delta_{a} \partial^{+}f(x,\zeta)$.
\end{proof}

\begin{Rem}\label{derivindependent}
Let $A\subset\vG$ be measurable, $f:A\to\vM$ be Lipschitz and $x\in A$. Suppose there exists \textbf{one} curve $\zeta_x^t\in\vG$ approximating direction $\zeta$ at $x$ with
\[ 
\lim_{t\downarrow 0} \delta_{1/t} \Big( f(x)^{-1}f(x\zeta_{x}^{t}) \Big) =  w\in\vM.
\]
Then for any other curve $\eta_x^t$ approximating direction $\zeta$ at $x$, we also have
\[
\lim_{t\downarrow 0} \delta_{1/t} \Big( f(x)^{-1}f(x\eta_{x}^{t}) \Big) =  
\lim_{t\downarrow 0} \delta_{1/t} \Big( f(x)^{-1}f(x\zeta_{x}^{t}) \Big) =  w\in\vM.
\]
In other words, $f$ is differentiable at $x$ in direction $\zeta$ and $w=\partial^{+}f(x,\zeta)$.
\end{Rem}

\begin{Rem}\label{Rem:derivopen}
In the simple case where $A\subset\vG$ is open, we may choose the simplest curve $\zeta_{x}^{t}=\delta_{t}\zeta$ for sufficiently small $t$. Hence, by Remark~\ref{derivindependent}, to verify directional differentiability
for a Lipschitz function $f:A\to\vM$, we have only to check the existence of the limit
\[ 
\lim_{t\downarrow 0} \delta_{1/t} \Big( f(x)^{-1}f(x\delta_{t}\zeta) \Big) =  \partial^{+}f(x,\zeta).
\]
\end{Rem}

Since for Lipschitz functions directional differentiability requires only one curve $\zeta_x^t$, it is useful to find those points of the domain where we have a natural choice.

\begin{Def}\label{equivDef}
Let $A\subset \vG$ and $x, \zeta\in \vG$. We say that $A$ is \emph{dense at $x$ in direction $\zeta$} if
\[
\frac{\mathcal{L}^1\pa{\{0<\theta<t:\ x\delta_{\theta} \zeta\notin A\}}}{t}\to 0 \quad \text{as }t\downarrow 0.
\]
\end{Def}

\begin{Pro}\label{Pro:fDiffDirZeta}
Suppose $A\subset\vG$ is measurable and dense at $x\in A$ in direction $\zeta$. Let $f:A\to\vM$ be Lipschitz. Then the following are true:
\begin{enumerate}
\item there exist $\zeta_{x}^{t}$ approximating in $A$ the direction $\zeta$ at $x$
\item defining $A(x,\zeta)=\{t\in\mathbb{R}\ |\ x\delta_t\zeta\in A\}$, we have $0\in\overline{A(x,\zeta)}$ 
\item $f$ is differentiable at $x$ in direction $\zeta$ if and only if the following limit exists:
\begin{align}\label{defMagnani}
\lim_{\;t\downarrow 0,\;t\in A(x, \zeta)}\delta_{1/t}(f(x)^{-1}f(x\delta_t\zeta)).
\end{align}
In either case, the above limit equals $\partial^+ f(x, \zeta)$.
\end{enumerate}
\end{Pro}

\begin{proof}
For $t>0$, choose $0<T(t)\le t$ such that $x\delta_{T(t)}\zeta\in A$ and
\[
t-T(t)<\inf\big \{ t-a\colon 0<a\le t, \, x\delta_{a}\zeta \in A\big\}+t^2.
\]
Density of $A$ at $x$ in direction $\zeta$ implies $T(t)/t\to 1$. The curve defined by $\zeta_{x}^{t}=\delta_{T(t)}\zeta$ then satisfies property (1).

Property (2) is clear from the definition of directional density.

Suppose now that the limit \eqref{defMagnani} exists. Choosing the previously defined curve $\zeta_x^t$, this gives the existence of the limit
\[
\lim_{t \downarrow 0}\delta_{\frac{1}{t}}(f(x)^{-1}f(x\zeta_x^t))=\lim_{t \downarrow 0}\delta_{\frac{T(t)}t}\pa{\delta_{\frac{1}{T(t)}}(f(x)^{-1}f(x\delta_{T(t)}\zeta))}.
\]
By Remark~\ref{derivindependent}, $f$ is differentiable at $x$ in direction $\zeta$ with $\partial^+ f(x, \zeta)$ equal to the limit in \eqref{defMagnani}.

Conversely, assume that $f$ is differentiable at $x$ in direction $\zeta$. Considering again the previous curve $\zeta_x^t=\delta_{T(t)}\zeta$, it follows that
\begin{align}
\partial^+ f(x, \zeta)=\lim_{\;t\downarrow 0,\;t\in A(x, \zeta)}\delta_{1/t}(f(x)^{-1}f(x\zeta_x^t))=
\lim_{\;t\downarrow 0,\;t\in A(x, \zeta)}
\pa{\delta_{1/t}(f(x)^{-1}f(x\delta_t\zeta))}E_{x,\zeta,t},
\end{align}
where we clearly have
\[
\rho\pa{E_{x,\zeta,t}}=\rho\pa{\delta_{1/t}(f(x\delta_t\zeta)^{-1}f(x\zeta_x^t))}\le \frac{L\, d(\delta_t\zeta,\zeta_x^t)}t\to 0.
\]
This proves property (3).
\end{proof}

Recall that a Banach space $X$ has the \emph{Radon-Nikodym property (RNP)} if every Lipschitz map $f\colon [0,1]\to X$ is differentiable almost everywhere. We say that a subspace of a Banach space has the RNP if it does when considered as a Banach space in its own right. To obtain directional derivatives we use the following theorem \cite[Theorem 3.1]{MagTap14}.

\begin{Teo}\label{Teo:1D_differentiability}
Suppose $H_1\subset \vM$ has the RNP. Let $A \subset \vR$ and $\gamma \colon A\to \vM$ be a Lipschitz mapping. Then $\gamma$ is almost everywhere differentiable.
\end{Teo}

\begin{Teo}\label{lineardensityderivatives}
Suppose $H_1\subset \vM$ has the RNP. Let $A\subset \vG$ be measurable and $\zeta$ be horizontal, namely $\zeta \in V_{1}\setminus\{0\}$. If $f\colon A \to \vM$ is Lipschitz, then for almost every point $x\in A$:
\begin{enumerate}
\item $A$ is dense at $x$ in direction $\zeta$,
\item $f$ is differentiable at $x$ in direction $\zeta$.
\end{enumerate}
\end{Teo}

\begin{proof}
Choose $\tilde V_1\subset V_1$ such that $\tilde V_1\oplus\mathrm{span}\{\zeta\}=V_1$ and consider a basis $(e_1,\ldots,e_{n-1})$ of
\[
N:=\tilde V_1\oplus V_2\oplus\cdots\oplus V_s,
\]
for which each $e_{i}$ belongs to some subspace $V_{j}$. The basis $(e_1,\ldots,e_{n-1},\zeta)$  of $\vG$ allows us to define 
$\Psi:\bbR^n\to\vG$ as follows
\[
\Psi(\xi,t)=\left(\sum_{j=1}^{n-1}\xi_j e_j\right)(t\zeta).
\]
The map $\Psi$ is a global diffeomorphism (see \cite[Proposition 7.6]{Mag}). By 3.1.3(5) of \cite{Fed69}, for a.e. $(\xi,t)\in \Psi^{-1}(A)$ the set of points 
\[
\{\theta\in\bbR: (\xi,\theta)\notin \Psi^{-1}(A)\}
\]
has density zero at $t$. Since $\zeta\in V_1$, hence $\Psi(\xi,t)=\Psi(\xi,0)\delta_t\zeta$, the previous statement implies that for a.e.\ $x=\Psi(\xi,t)\in A$, the set $A$ is dense at $x$ in direction $\zeta$. We denote the set of these points by $A_\zeta$, observing that is measurable and $\mu(A\sm A_\zeta)=0$.

Let $D_{f,\zeta}\subset A_\zeta$ be the set of points at which $f$ is differentiable in direction $\zeta$. Our proof is completed once we have $\mu(A_\zeta\setminus D_{f,\zeta})=0$, that is $\cL^n\pa{\Psi^{-1}(A_\zeta\sm D_{f,\zeta})}$. By the measurability of $D_{f,\zeta}$ (whose proof is technical and postponed to Section~\ref{Sect:Measurability}), the set $Z_{f,\zeta}=A_\zeta\sm D_{f,\zeta}$ is also measurable. 
We can then apply Fubini's theorem to the measurable set 
\[
\Psi^{-1}(Z_{f,\zeta}),
\]
whose measure can be recovered by integration of measures of the 
$1$-dimensional sections 
\[
\pa{\Psi^{-1}(Z_{f,\zeta})}_\xi=\Psi(\xi,\cdot)^{-1}(Z_{f,\zeta}).
\]
where $\xi\in\bbR^{n-1}$. The composition $t\to f(\Psi(\xi,t))\in\vM$ is Lipschitz on $\Psi(\xi,\cdot)^{-1}(A_\zeta)\subset\bbR$, therefore Theorem~\ref{Teo:1D_differentiability} provides its a.e.\ differentiability on $\Psi(\xi,\cdot)^{-1}(A_\zeta)\subset\bbR$. It follows that $\Psi(\xi,\cdot)^{-1}(Z_{f,\zeta})$ is negligible, hence so is $\Psi^{-1}(Z_{f,\zeta})$.
\end{proof}

\begin{Def}\label{h-hom}
A \emph{homogeneous homomorphism}, in short \emph{h-homomorphism}, from $\vG$ to $\vM$ is a map $L\colon \vG\to \vM$ such that $L(xy)=L(x)L(y)$ and $L(\delta_{r}(x))=\delta_r(L(x))$ for all $x,y\in \vG$ and $r>0$.
\end{Def}


Any h-homomorphism is automatically Lipschitz. Indeed, since $\vG$ is stratified there exist a positive integer $N$ and 
$v_1,\ldots, v_N\in V_1$ such that for some $T>0$ the set
\[
V:=\{\delta_{t_{1}}v_1\cdots \delta_{t_N}v_N\colon |t_i|<T\}
\]
contains the unit ball. The h-homomorphism property then yields the Lipschitz continuity. The reader may also see \cite[Proposition 3.11]{Mag01}, which is stated there for stratified group targets, but works equally well for Banach group targets.

\begin{Def}\label{def:hdiff}\rm 
Let $A\subset\vG$ and $x\in A\cap D(A)$. We say that $f\colon A\lra \vM$ is {\em h-differentiable} at $x$, or simply {\em differentiable} at $x$, if there exists a h-homomorphism $L\colon\vG\lra\vM$ such that
\[ \frac{\rho(f(x)^{-1}f(xz),L(z))}{d(z)}\to 0 \qquad \mbox{as $z\in x^{-1}A$ and $d(z)\downarrow 0$}.\]
The mapping $L$ is unique and called the \emph{h-differential} of $f$ at $x$.
\end{Def}

\subsection{Porous sets}

We now define porous sets and $\sigma$-porous sets.

\begin{Def}\label{def_porous}
Let $(M, \rho)$ be a metric space, $E\subset M$ and $a\in M$. We say that  $E$ is \emph{porous at $a$} if there exist $\Lambda>0$ and a sequence $x_{n}\to a$ such that
\[B(x_n,\Lambda \rho(a,x_n))\cap E=\emptyset \qquad \mbox{for every $n\in \mathbb{N}$}.\] 

A set $E$ is \emph{porous} if it is porous at each point $a\in E$ with $\Lambda$ independent of $a$. A set is \emph{$\sigma$-porous} if it is a countable union of porous sets. 
\end{Def}

Porous sets in stratified groups have measure zero. This follows from the fact that Haar measure on stratified groups is Ahlfors regular, hence doubling, so the Lebesgue differentiation theorem applies \cite[Theorem 1.8]{Hei01}.

\section{Directional derivatives act as h-homomorphisms outside a $\sigma$-porous set}\label{subsets}
In this section we study applications of porosity to directional derivatives. Recall that $\vG$ is a stratified group of step $s$ and $\vM$ is a Banach homogeneous group of step $\nu$. We will need the following estimate for the Banach homogeneous distance in $\vM$.

\begin{Lem}\label{estMagn}
Let $N\in\N$ and $A_j,B_j\in\vM$ for $j=1,\ldots,N$. Suppose there exists $b>0$ such that 
\[
\rho(B_jB_{j+1}\cdots B_N)\leq b\quad\text{ and }\quad \rho(A_j,B_j)\leq b \quad \mbox{for }j=1,\ldots,N.
\]
Then there exists $C_b>0$, depending on $b$, such that
\begin{eqnarray}\label{estimprod}
\rho\big(A_1A_2\cdots A_N,B_1B_2\cdots B_N\big)\leq C_b
\sum_{j=1}^N \rho (A_j,B_j)^{1/\nu}.
\end{eqnarray}
\end{Lem}

The finite dimensional version of Lemma \ref{estMagn} can be found in \cite[Lemma 3.7]{Mag}. The proof works in a similar way for Banach homogeneous groups. However, due to the key role played by Lemma~\ref{estMagn} in our arguments,
we will present its proof in Section~\ref{Sect:estMagn}. Since any stratified group is a Banach homogeneous group, we can also apply Lemma \ref{estMagn} in $\vG$, replacing $\nu$ by $s$ and $\rho$ by $d$.

We will also need the following estimate for distances in stratified groups \cite[Lemma 2.13]{FS}. 

\begin{Lem}\label{stimagroup}
There is a constant $D>0$ such that
\begin{align*}
d(x^{-1}yx) \leq D\,\Big(d(y)+ d(x)^{\frac{1}{s}}d(y)^{\frac{s-1}{s}}+d(x)^{\frac{s-1}{s}}d(y)^{\frac{1}{s}}\Big)\quad \mbox{for }x,y\in\vG.
\end{align*}
\end{Lem}

Unless otherwise stated, we denote by $f\colon A\to \vM$ a fixed Lipschitz function on a measurable set $A\subset\vG$ with Lipschitz constant $L$.

\begin{Def}\label{Def:P}
Fix $\zeta, \eta \in \vG$, $y,z\in \vM$ and $\varepsilon, \delta>0$. Let $C_{1}$ and $C_{2}$ respectively be the constants obtained from applying Lemma~\ref{estMagn} in $\vM$ with $b=\max\{\varepsilon,\ \rho(y)+\rho(z)\}$ and in $\vG$ with $b=\max\{\varepsilon,\ \varepsilon/L,\ d(\zeta)+d(\eta)\}$.

We define $P(\zeta, \eta, y, z, \varepsilon, \delta,A)$ to be the set of points $p\in A$ for which the following properties hold:
\begin{itemize}
\item For all $0<t<\delta$ there exist $\zeta_p^t, \eta_p^t \in \vG$ satisfying
\begin{equation}\label{A1}
p\zeta_p^t\in A \quad \mbox{and} \quad d(\zeta_p^t, \delta_t \zeta)<\varepsilon t ,
\end{equation}
\begin{equation}\label{A2}
p\eta_p^t\in A \quad \mbox{and}\quad d(\eta_p^t, \delta_t \eta)<\varepsilon t,
\end{equation}
\begin{equation}\label{A3}
\rho(f(p)^{-1}f(p \zeta_p^t),\delta_t y)\leq \varepsilon t,
\end{equation}
\begin{equation}\label{A4}\rho(f(p)^{-1}f(p\eta_p^t), \delta_t z)\leq \varepsilon t.
\end{equation}
\item For arbitrarily small $t$ there exist $\omega_p^t\in \vG$ such that
\begin{equation}\label{badt}
p\omega_p^t\in A \quad \mbox{and} \quad d(\omega_p^t, \delta_t (\zeta\eta))<\varepsilon t, 
\end{equation}
\begin{equation}\label{badt2}
\rho(f(p)^{-1} f(p\omega_p^t),\delta_t (yz))> \left( 3C_{1}\varepsilon^{1/\nu}+ L\varepsilon +  LC_{2} \left(2+L^{-1/s}\right)\varepsilon^{1/s} \right)t.
\end{equation}
\end{itemize}
\end{Def}

Intuitively, $P(\zeta, \eta, y, z, \varepsilon, \delta,A)$ consists of points for which the directional derivatives in direction $\zeta$ and $\eta$ look like $y$ and $z$ respectively, at scales less than $\delta$ and with accuracy $\varepsilon$.

Notice that provided $\varepsilon$ and $\delta$ are bounded by some fixed constant $K$, the constants $C_{1}$ and $C_{2}$ in Definition~\ref{Def:P} are bounded by constants independent of the precise value of $\varepsilon$ and $\delta$.

\begin{Lem}\label{lemmageneral}
The set $P(\zeta, \eta, y, z, \varepsilon, \delta,A)$ is porous.
\end{Lem}

\begin{proof}
For this proof we abbreviate $P=P(\zeta,\eta,y,z,\varepsilon,\delta,A)$. Let $x\in P$ and choose $0<t<\delta/2$ for which there exist $\omega_{x}^{t}\in \vG$ satisfying \eqref{badt} and \eqref{badt2}. Fix also $\zeta_{x}^{t}$ satisfying \eqref{A1} and \eqref{A3}. Since $x$ was any element of $P$ and $t$ could be chosen arbitrarily small, to prove $P$ is porous it suffices to show that $B(x\zeta_x^t,\varepsilon t/L)\cap P=\emptyset$. 

We suppose $B(x\zeta_x^t,\varepsilon t/L)\cap P\neq \emptyset$ and deduce a contradiction. Choose $p\in B(x\zeta_x^t,\varepsilon t/L)\cap P$. Use the definition of $P$ to choose $\eta_p^t\in\vG$
satisfying \eqref{A2} and \eqref{A4}. We first estimate
\[
\rho(f(x)^{-1}f(p \eta_p^t), \delta_t(yz)).
\]
To this end, observe we can write
\[\delta_{1/t}(f(x)^{-1}f(p \eta_p^t))=A_1 A_2 A_3, \qquad \quad  yz= B_1B_2B_3,\]
where
\[A_1=\delta_{1/t}(f(x)^{-1} f(x\zeta_x^t)), \quad A_2=\delta_{1/t}(f(x\zeta_x^t)^{-1} f(p)), \quad A_3=\delta_{1/t}(f(p)^{-1} f(p\eta_p^t)),\]
and
\[B_1=y,\quad B_2=0,\quad B_3=z.\]
We check $A_i, B_i$ satisfy the hypotheses of Lemma~\ref{estMagn} with $b=\max\{\varepsilon,\ \rho(y)+\rho(z)\}$. Using \eqref{A3} gives
\begin{align*}
\rho(A_1, B_1)&= \rho(\delta_{1/t}(f(x)^{-1} f(x\zeta_x^t)), y)= \frac{1}{t}\rho(f(x)^{-1} f(x\zeta_x^t), \delta_t y)\leq \varepsilon.
\end{align*}
Recalling that $p\in B(x\zeta_x^t,\varepsilon t/L)$, we obtain
\begin{align*}
\rho(A_2,B_2)&=\rho(\delta_{1/t}(f(x\zeta_x^t)^{-1} f(p)), 0)=\frac{1}{t}\rho( f(p),f(x\zeta_x^t))\leq \frac{L}{t} d(x\zeta_x^t, p)\leq \varepsilon.
\end{align*}
Using \eqref{A4} also gives
\begin{align*}
\rho(A_3, B_3)=\rho(\delta_{1/t}(f(p)^{-1} f(p\eta_p^t)), z)= \frac{1}{t}\rho(f(p)^{-1} f(p\eta_p^t), \delta_t z)\leq \varepsilon.
\end{align*}
Hence $\rho(A_i, B_i)\leq \varepsilon$ for each $i$. Clearly also $\rho(B_i B_{i+1} \cdots B_3)\leq \rho(y)+\rho(z)$ for each $i$. Hence the hypotheses of Lemma~\ref{estMagn} are satisfied for our choice of $b$ and we get
\begin{equation}\label{eqomega}
\rho(f(x)^{-1}f(p \eta_p^t), \delta_t(yz))\leq C_{1}t\sum_{i=1}^3 d(A_i, B_i)^{1/\nu} \leq 3C_{1}\varepsilon^{1/\nu} t.
\end{equation}

Now let $\theta_{x}^t=x^{-1}p\eta_p^t$. Then $x\theta_{x}^t=p\eta_p^t\in A$ and \eqref{eqomega} gives
\begin{equation}\label{Jan19}
\rho(f(x)^{-1}f(x\theta_{x}^t), \delta_t(yz))\leq 3C_{1}\varepsilon^{1/\nu} t.
\end{equation}
Now we estimate $d(\theta_{x}^t, \delta_t (\zeta \eta))$. Since $p\in B(x\zeta_x^t, \varepsilon t/L)$, we may choose $h^t\in \vG$ with $d(h^t)\leq \varepsilon t/L$ and $p=x\zeta_x^t h^t$. Therefore
\[d(\theta_{x}^t, \delta_t (\zeta \eta))=d(x^{-1}p\eta_p^t, \delta_t (\zeta \eta))=d(\zeta_x^t h^t\eta_p^t, \delta_t (\zeta \eta)).\]
Now define
\[A_1=\delta_{1/t}\zeta_x^t, \qquad A_2=\delta_{1/t}h^t, \qquad A_3=\delta_{1/t}\eta_p^t,\]
and
\[B_1=\zeta, \qquad B_2=0, \qquad B_3=\eta.\]
Let $b=\max\{\varepsilon,\ \varepsilon/L,\ d(\zeta)+d(\eta)\}$. Using \eqref{A1} and \eqref{A2} shows $d(A_{1},B_{1})$ and $d(A_{3},B_{3})$ are bounded by $\varepsilon$. Using $d(h^t)\leq \varepsilon t/L$ gives $d(A_{2},B_{2})\leq \varepsilon/L$. Clearly $d(B_{i}B_{i+1}\cdots B_{3})\leq d(\zeta)+d(\eta)$ for each $i$. Hence the hypotheses of Lemma~\ref{estMagn} are satisfied, giving
\begin{align*}
d(\theta_{x}^t,  \delta_t(\zeta \eta))\leq C_{2} t\sum_{i=1}^3 \rho(A_i, B_i)^{1/s}
\leq C_{2} \left(2+L^{-1/s}\right)\varepsilon^{1/s}t.
\end{align*}
Recall $\omega_{x}^{t}$ were chosen satisfying \eqref{badt} and \eqref{badt2}. Using also \eqref{Jan19}, we estimate as follows:
\begin{align*}
\rho(f(x)^{-1}f(x \omega_{x}^t), \delta_t(yz))&\leq \rho(f(x)^{-1}f(x \theta_{x}^t), \delta_t(yz))+\rho(f(x \omega_{x}^t), f(x \theta_{x}^t))\\
&\leq 3C_{1}\varepsilon^{1/\nu} t + Ld(\omega_{x}^t, \theta_{x}^t)\\
&\leq 3C_{1}\varepsilon^{1/\nu} t + L\left(d(\omega_{x}^t, \delta_t(\zeta\eta))+ d(\theta_{x}^t, \delta_t(\zeta\eta))\right)\\
&\leq 3C_{1}\varepsilon^{1/\nu} t+ L\varepsilon t + LC_{2} \left(2+L^{-1/s}\right)\varepsilon^{1/s}t \\
&= \left( 3C_{1}\varepsilon^{1/\nu}+ L\varepsilon + LC_{2} \left(2+L^{-1/s}\right)\varepsilon^{1/s} \right)t.
\end{align*}
This contradicts the choice of $\omega_{x}^{t}$, forcing us to conclude $B(x\zeta_x^t,\varepsilon t/L)\cap P=\emptyset$. Hence $P$ is porous, which concludes the proof.
\end{proof}

We now prove Theorem \ref{generalcase} by putting together countably many of the sets from Definition~\ref{Def:P}.

\begin{proof}[Proof of Theorem \ref{generalcase}]
Using Theorem~\ref{Teo:separable_image} we can assume that the Banach homogeneous group $\vM$ is separable. Let $P$ be the countable union of sets $P(\zeta, \eta, y, z, \varepsilon, \delta,A)$ for $\zeta, \eta$ in a countable dense subset $\cN_{\vG}$ of $\vG$, $y,z$ in a countable dense subset $\cN_{\vM}$ of $\vM$, and $\varepsilon, \delta$ positive rationals. By Proposition~\ref{dilatedirectional}, it is enough to show that for points $x\in A\cap D(A) \setminus P$ we have
\[
\partial^{+}f(x,\zeta \eta)=\partial^{+}f(x,\zeta)\partial^{+}f(x,\eta)
\]
whenever $\zeta, \eta \in \vG$ for which $ \partial^{+}f(x,\zeta), \partial^{+}f(x,\eta)$ exist.

Suppose $x\in A\cap D(A)\setminus P$ and $\partial^{+}f(x,\zeta), \partial^{+}f(x, \eta)$ exist for some $\zeta, \eta \in \vG$. Fix $\varepsilon>0$ rational. Choose $\zeta_{x}^{t}, \eta_{x}^{t}\in \vG$ for $t>0$ with $x\zeta_{x}^{t}, x\eta_{x}^{t}\in A$ and $d(\zeta_{x}^{t},\delta_{t}\zeta)/t, d(\eta_{x}^{t},\delta_{t}\eta)/t \to 0$. Using the existence of $\partial^{+}f(x,\zeta), \partial^{+}f(x,\eta)$, choose $\delta>0$ rational such that for $0<t<\delta$:
\[d(\zeta_{x}^{t},\delta_{t}\zeta)< \varepsilon t/2 \quad \mbox{and} \quad \rho( \delta_{1/t}(f(x)^{-1}f(x\zeta_{x}^{t})), \partial^{+}f(x,\zeta))\leq  \varepsilon/2,\]
\[d(\eta_{x}^{t},\delta_{t}\eta)< \varepsilon t /2 \quad \mbox{and} \quad \rho( \delta_{1/t}(f(x)^{-1}f(x\eta_{x}^{t})), \partial^{+}f(x,\eta))\leq \varepsilon/2.\]
Now choose $y, z\in \cN_{\vM}$ such that
\begin{equation}\label{eq:estimatesC_1varepsilon}
 \rho( \partial^{+}f(x,\zeta), y)<\varepsilon/2, \quad \rho( \partial^{+}f(x,\eta), z)<\varepsilon/2, \quad \rho( \partial^{+}f(x,\zeta)\,\partial^{+}f(x,\eta), yz)<C_{1}\varepsilon^{1/\nu}. 
\end{equation}
In particular, we have
\begin{equation}\label{Jan19.1}
\rho( \delta_{1/t}(f(x)^{-1}f(x\zeta_{x}^{t})), y)\leq \varepsilon \quad \mbox{and} \quad \rho( \delta_{1/t}(f(x)^{-1}f(x\eta_{x}^{t})), z)\leq \varepsilon \quad \mbox{ for }0<t<\delta.
\end{equation}
Now choose $\zeta', \eta'\in\cN_{\vG}$ such that $d(\zeta, \zeta')<\varepsilon /2$, $d(\eta, \eta')<\varepsilon /2$ and $d(\zeta\eta, \zeta'\eta')<\varepsilon/2$.
Then
\begin{equation}\label{Jan19.2}
d(\zeta_{x}^{t},\delta_{t}\zeta')<\varepsilon t \quad \mbox{and} \quad d(\eta_{x}^{t},\delta_{t}\eta')<\varepsilon t \quad \mbox{ for }0<t<\delta.
\end{equation}
Equations \eqref{Jan19.1} and \eqref{Jan19.2} show that the point $x$ satisfies the first four conditions of the set $P(\zeta',\eta',y,z,\varepsilon,\delta,A)$, i.e. \eqref{A1}--\eqref{A4} hold with $\zeta, \eta$ replaced by $\zeta', \eta'$
and $p$ replaced by $x$. 

Since $x\notin P\supset P(\zeta',\eta',y,z,\varepsilon,\delta,A)$, the analogue of \eqref{badt} or \eqref{badt2} for the set $P(\zeta',\eta',y,z,\varepsilon,\delta,A)$ must fail for sufficiently small $t$. Hence whenever $t$ is sufficiently small and $\omega_{x}^{t}$ satisfy 
the two conditions 
\begin{equation}\label{eq:xtA}
x\omega_{x}^t\in A \quad\text{ and } \quad d(\omega_{x}^t, \delta_t (\zeta' \eta'))<\varepsilon t,
\end{equation}
then we must have
\begin{equation}\label{touse}
\rho(f(x)^{-1} f(x\omega_{x}^t),\delta_t (yz))\leq  \left( 3C_{1}\varepsilon^{1/\nu}+ L\varepsilon + LC_{2} \left(2+L^{-1/s}\right)\varepsilon^{1/s} \right)t.
\end{equation}
Let $\omega_{x}^{t}$ approximate in $A$ the direction $\zeta\eta$ at $x$. Then $x\omega_{x}^t\in A$ and $d(\omega_{x}^t, \delta_t (\zeta \eta))<\varepsilon t/2$ for sufficiently small $t$. Since $d(\zeta\eta, \zeta'\eta')<\varepsilon/2$ it follows that $\omega_{x}^{t}$ satisfy \eqref{eq:xtA} and hence \eqref{touse} for sufficiently small $t$. Taking into account the third inequality of \eqref{eq:estimatesC_1varepsilon}, we obtain for sufficiently small $t$:
\begin{equation}\label{Jan19.3}
\rho(f(x)^{-1} f(x\omega_{x}^t),\delta_t (\partial^{+}f(x,\zeta)\partial^{+}f(x,\eta)) )\leq \left( 4C_{1}\varepsilon^{1/\nu}+ L\varepsilon +  LC_{2} \left(2+L^{-1/s}\right)\varepsilon^{1/s} \right)t.
\end{equation}
To summarize, if $\omega_{x}^{t}$ approximate the direction $\zeta\eta$ at $x$ then \eqref{Jan19.3} holds for sufficiently small $t$. Provided $\varepsilon, \delta<1$, the constants $C_{1}$ and $C_{2}$ are independent of the precise value of $\varepsilon$ and $\delta$. We conclude that $\partial^{+}f(x,\zeta\eta)$ exists and is equal to $\partial^{+}f(x,\zeta)\partial^{+}f(x,\eta)$.
\end{proof}

\section{From directional derivatives to differentiability}\label{diff}

We now prove Theorem \ref{precisesubset} which gives conditions for pointwise differentiability. Combining these results with those of Section~\ref{subsets} will yield Theorem~\ref{thm:Pansudiff}.

\begin{proof}[Proof of Theorem \ref{precisesubset}]
First we suppose $f$ is differentiable at $x\in A\cap D(A)$. Then there exists a h-homomorphism $L_x:\vG\to \vM$ such that
\begin{align}\label{di}
\rho(f(x)^{-1} f(xz), L_x(z))=o(d(z)) \qquad \mbox{ as }z\in x^{-1}A \quad \mbox{and} \quad z\to 0.
\end{align}
Let $\zeta\in \vG$. Using Lemma~\ref{convzeta}, choose $\zeta_x^{t}$ such that $x\zeta_{x}^{t}\in A$ and $d(\zeta_{x}^{t}, \delta_{t}\zeta)/t\to 0$ as $t\downarrow 0$. Then \eqref{di} gives
\begin{align}\label{dww}
\rho\left(f(x)^{-1}f(x\zeta_x^{t}), L_x(\zeta_x^{t})\right)=o(d(\zeta_x^{t})) \qquad \mbox{as $t\downarrow 0$}.
\end{align}
Since $d(\zeta_{x}^{t}, \delta_{t}\zeta)/t\to 0$, there exists $K=K(\zeta)>0$ such that $d(\zeta_x^{t})< K t$ for sufficiently small $t$. Combining this with \eqref{dww} gives
\begin{align}\label{dd}
\lim_{t\downarrow 0}\rho\left(\delta_{1/t}(f(x)^{-1}f(x\zeta_x^{t})), \delta_{1/t}L_x(\zeta_x^{t})\right)=0.
\end{align}
Since $L_x$ is a h-homomorphism (and hence Lipschitz), we have
\[
\rho(L_{x}(\zeta),\delta_{1/t}L_{x}(\zeta_{x}^{t}))=\rho(L_{x}(\zeta),L_{x}(\delta_{1/t}\zeta_{x}^{t})) \leq \mathrm{Lip}(L_{x})\frac{d(\delta_{t}\zeta,\zeta_{x}^{t})}{t}\to 0.
\]
Combining this with \eqref{dd} shows $\partial^+ f(x,\zeta)$ exists and is equal to $L_x(\zeta)$ for any $\zeta \in \vG$. Since $L_{x}$ is a h-homomorphism, the h-homomorphism property of directional derivatives follows.

For the converse statement, we assume $\partial^+ f(x, \zeta)$ exists for any $\zeta \in \vG$ and $L_x:\vG\to \vM$ defined by $L_x(\zeta):=\partial^+ f(x, \zeta)$ is a h-homomorphism. For every $\zeta \in \vG$ with $d(\zeta)\leq 1$, use Lemma \ref{convzeta} to choose $\zeta_{x}^{t}$ for $t>0$ with $x\zeta_{x}^{t}\in A$ and $d(\zeta_{x}^{t},\delta_{t}\zeta)/t \to 0$ as $t\downarrow 0$ uniformly for $\zeta \in \vG$ with $d(\zeta)\leq 1$. We first show that
\begin{equation}\label{eq:unif_part_deriv}
\delta_{1/t}(f(x)^{-1}f(x\zeta_x^t)) \to \partial^+ f(x,\zeta)
\quad\text{ as $t\downarrow 0$ uniformly for $\zeta \in \vG$ with $d(\zeta)\leq 1$.}
\end{equation}

Let $K=1+\mathrm{Lip}(L_x)+3L>0$ and fix $\varepsilon>0$. Choose a finite set $S\subset \vG$ such that for any $\eta \in \vG$ with $d(\eta)\leq 1$, there exists $\zeta \in S$ with $d(\zeta, \eta)< \varepsilon/K$. Choose $\delta>0$ such that 
\begin{align}\label{terzastima}
\rho\left(\delta_{1/t}(f(x)^{-1}f(x\zeta_{x}^t)),\partial^+f(x,\zeta)\right)<\varepsilon/K \qquad \mbox{for every }\zeta \in S \mbox{ and }0<t<\delta,
\end{align}
and
\begin{align}\label{may31}
d(\eta_{x}^t, \delta_t\eta)/t< \varepsilon/K \qquad \mbox{for every }\eta \in \vG \mbox{ with }d(\eta)\leq 1 \mbox{ and }0<t<\delta.
\end{align}

Given $\eta \in \vG$ with $d(\eta)\leq 1$, choose $\zeta \in S$ with $d(\zeta, \eta)< \varepsilon/K$. We observe that
\begin{align}\label{stimazero}
\rho\left(\delta_{1/t}(f(x)^{-1}f(x\eta_{x}^t)),\partial^+f(x,\eta)\right)&\leq  \rho\left(\delta_{1/t}(f(x)^{-1}f(x\zeta_{x}^t)),\partial^+f(x,\zeta)\right)\\
\nonumber
&+\rho\left(\delta_{1/t}(f(x)^{-1}f(x\eta_{x}^t)), \delta_{1/t}(f(x)^{-1}f(x\zeta_{x}^t))\right)\\
\nonumber
&+\rho\left(\partial^+f(x,\zeta),\partial^+f(x,\eta)\right).
\end{align}
The first term in \eqref{stimazero} is estimated by \eqref{terzastima} for $0<t<\delta$. We estimate the second term of \eqref{stimazero} as follows:
\begin{align*}
\rho\left(\delta_{1/t}(f(x)^{-1}f(x\eta_{x}^t)), \delta_{1/t}(f(x)^{-1}f(x\zeta_{x}^t))\right)&=\frac{1}{t}\rho\left(f(x\eta_{x}^t), f(x\zeta_{x}^t)\right)\\
&\leq \frac{L}{t} d(\eta_{x}^t, \zeta_{x}^t)\\
&\leq L \left(\frac{d(\eta_{x}^t, \delta_t\eta)}{t}+\frac{d(\zeta_{x}^t, \delta_t\zeta)}{t}+d(\zeta,\eta)\right).
\end{align*}
Using our choice of $\zeta$ and \eqref{may31}, we obtain
\begin{equation}\label{may31.2}
\rho\left(\delta_{1/t}(f(x)^{-1}f(x\eta_{x}^t)), \delta_{1/t}(f(x)^{-1}f(x\zeta_{x}^t))\right)< 3L\varepsilon/K \qquad \mbox{for }0<t<\delta.
\end{equation}
We estimate the final term in \eqref{stimazero} as follows:
\begin{align}\label{secondastima}
\rho(\partial^+f(x,\zeta), \partial^+f(x,\eta))=\rho(L_x(\zeta), L_x(\eta))\leq \mathrm{Lip}(L_x) d(\zeta,\eta)< \varepsilon \mathrm{Lip}(L_x)/K.
\end{align}
Using \eqref{terzastima}, \eqref{may31.2} and \eqref{secondastima} in \eqref{stimazero} together with the definition of $K$ yields
\begin{align*}
\rho\left(\delta_{1/t}(f(x)^{-1}f(x\eta_{x}^t)),\partial^+f(x,\eta)\right)< \varepsilon.
\end{align*}
This proves the uniform convergence \eqref{eq:unif_part_deriv}. 

To conclude we will show that $\rho(f(x)^{-1}f(xz), L_x(z))=o(d(z))$ as $z\to 0$ with $z\in x^{-1}A$. First notice that for $\zeta\in \vG$ with $d(\zeta)\leq 1$:
\begin{align*}
\frac{1}{t}\rho(f(x)^{-1}f(x\zeta_x^t), L_x(\zeta_x^t))&=\rho(\delta_{1/t}(f(x)^{-1}f(x\zeta_x^t)), \delta_{1/t}L_x(\zeta_x^t))\\
&\leq \rho(\delta_{1/t}(f(x)^{-1}f(x\zeta_x^t)),\partial^+f(x,\zeta))+\rho(\partial^+f(x,\zeta), \delta_{1/t}L_x(\zeta_x^t))\\
&= \rho(\delta_{1/t}(f(x)^{-1}f(x\zeta_x^t)),\partial^+f(x,\zeta))+\rho(L_x(\zeta),L_x(\delta_{1/t}\zeta_x^t))\\
&\leq \rho(\delta_{1/t}(f(x)^{-1}f(x\zeta_x^t)),\partial^+f(x,\zeta))+\mathrm{Lip}(L_x)d(\zeta,\delta_{1/t}\zeta_x^t).
\end{align*}
Hence \eqref{eq:unif_part_deriv} and our choice of $\zeta_{x}^{t}$ gives
\begin{align}\label{seconduniform}
\frac{1}{t}\rho(f(x)^{-1}f(x\zeta_x^t), L_x(\zeta_x^t))\to 0 \qquad \mbox{as }t\downarrow 0 \mbox{ uniformly for }d(\zeta)\leq 1.
\end{align}
Given $\varepsilon>0$, use \eqref{seconduniform} and our choice of $\zeta_{x}^{t}$ to choose $\delta>0$ such that
\[ \frac{d(\delta_{t}\zeta, \zeta_{x}^{t})}{t}<\varepsilon \quad \mbox{and} \quad \frac{\rho(f(x)^{-1}f(x\zeta_{x}^{t}),L_{x}(\zeta_{x}^{t}))}{t} < \varepsilon \qquad \mbox{for }d(\zeta)\leq 1 \mbox{ and }0<t<\delta.\]
Now let $z\in x^{-1}A$ with $d(z)\leq \delta$. Choose $\zeta\in \vG$ with $d(\zeta)=1$ such that $z=\delta_{t}\zeta$ for some $0<t<\delta$. We then estimate as follows:
\begin{align*}
\frac{\rho(f(x)^{-1}f(xz), L_{x}(z))}{d(z)} &\leq \frac{\rho(f(x)^{-1}f(x\delta_{t}\zeta), f(x)^{-1}f(x\zeta_{x}^{t}))}{t} + \frac{\rho(f(x)^{-1}f(x\zeta_{x}^{t}), L_{x}(\zeta_{x}^{t}))}{t}\\
& \qquad + \frac{\rho(L_{x}(\zeta_{x}^{t}),L_{x}(\delta_{t}\zeta))}{t}\\
&\leq (L+\mathrm{Lip}(L_{x})) \frac{d(\delta_{t}\zeta,\zeta_{x}^{t})}{t} +  \frac{\rho(f(x)^{-1}f(x\zeta_{x}^{t}), L_{x}(\zeta_{x}^{t}))}{t}\\
&< (L+\mathrm{Lip}(L_{x}) + 1)\,\varepsilon.
\end{align*}
This shows that $f$ is differentiable at $x$.

\end{proof}

Using our results above, we can now quickly show how Theorem~\ref{precisesubset} gives Theorem~\ref{thm:Pansudiff}. By \cite[Lemma 4.9]{Mag}, there exists a {\em spanning set of directions} $\nu_1,\ldots, \nu_N\in V_1$ for some $N\in\N$, namely there exists $T>0$ such that
\[
\set{\delta_{t_1}v_1\delta_{t_2}v_2\cdots\delta_{t_N}v_N: 0\le t_i<T}
\]
contains the unit ball of $\vG$.
If $H_1\subset \vM$ has the RNP then Theorem~\ref{lineardensityderivatives} yields a null set $Z\subset A$ such that for every $x\in A\setminus Z $: $A$ is dense at $x$ in direction $\nu_i$ and $\partial^+ f(x,\nu_i)$ exists for every $i=1,\ldots,N$. 
By Theorem~\ref{generalcase}, we can find a $\sigma$-porous set $P\subset \vG$ such that directional derivatives at $x$ act as h-homomorphisms whenever $x\in A\cap D(A)\setminus P$.

Combining these with the spanning property of $v_1,\ldots,v_N$, it follows that $f$ is differentiable at each $x\in A\cap D(A)\setminus (P\cup Z)$ in any direction $\zeta$ of the unit ball. Proposition~\ref{dilatedirectional} extends this directional differentiability to all directions $\zeta$ of $\vG$. The same Theorem~\ref{generalcase} shows that for $x\in A\cap D(A)\setminus (P\cup Z)$ the directional derivative $\vG\ni \zeta\to \partial^+f(x,\zeta)$ defines a h-homomorphism. We have established both conditions 1 and 2 of Theorem~\ref{precisesubset}, therefore $f$ is differentiable at every point of $A\cap D(A)\setminus (P\cup Z)$. Since $\mathcal{H}^{Q}_d(A\setminus D(A))=0$ and porous sets have measure zero, our claim follows.

\begin{Rem}
Taking into account Remark~\ref{Rem:derivopen}, when the domain $A$ of $f$ is an open set, properties \eqref{A1} and \eqref{A2} of Definition~\ref{Def:P} are automatically satisfied. Hence several arguments in the proofs of both Theorem~\ref{generalcase} and Theorem~\ref{precisesubset} become simpler.
\end{Rem}

\section{Appendix}\label{ProofMag} 

\subsection{Separable Banach Homogeneous Groups}\label{SepBanTarg}
Let $\vM$ be a Banach homogeneous group of step $\nu$.
For every $\cv=(v_1,\ldots,v_k)\in \vM^k$, $k\ge1$, we define the nonassociative product
\[
\cp_k(\cv)=[\cdots[[v_1,v_2],v_3]\cdots],v_k]=v_1\circ v_2\circ\cdots\circ v_k,
\]
where $\cp_k:\vM^k\to \vM$ is clearly a continuous multilinear mapping.
For any countable subset 
\[
\cN=\left\{x_j\in\vM:j\in\N\right\},
\]
we define its {\em rational homogeneous span} as follows
\[
\lan\lan \cN\ran\ran=\left\{ \sum_{j=0}^n  \sum_{l=1}^\nu \lambda_{jl}\, \pi_l(x_j): \ \lambda_{jl}\in\Q,\ n\in\N\right\}\subset\vM.
\]
The {\em closed homogeneous span of $\cN$} is simply $\lan\cN\ran=\overline{\lan\lan\cN\ran\ran}\subset\vM$.
This is clearly a linear subspace of $\vM$ that is also {\em homogeneous}, that is
$\delta_r\lan\cN\ran\subset\lan\cN\ran$ for every $r>0$. 
Since $\lan\lan\cN\ran\ran$ is countable, $\lan \cN\ran$ is a separable homogeneous Banach space,
that in particular contains $\overline\cN$.

\begin{Teo}\label{teo:sepsub}
There exists a separable Banach homogeneous subgroup $\vV\subset\vM$ such that $\lan \cN\ran\subset\vV$.
\end{Teo}

\begin{proof}
Let $\vY_1=\lan\cN\ran$ be the separable and homogeneous Banach space spanned by $\cN$ and consider 
\[
\vY_k=\left\{\sum_{\cv\in \cI} \cp_k(\cv):\ \cI\subset\vY_1^k\ \text{is finite}\right\}
\]
for each $k=2,\ldots,\nu$. Since dilations are also Lie algebra homomorphisms
and $\vY_1$ is closed under dilations, it follows that $\delta_r\vY_k\subset\vY_k$ for each $k=1,\ldots,\nu$ and $r>0$. 
It is clear that $\vY_k$ is a linear subspace also for $k=2,\ldots,\nu$.

We wish to show that for any couple of finite subsets
$\cI\subset \vY_1^k$ and $\cJ\subset \vY_1^l$ with 
\[
\sum_{\cv\in \cI} \cp_k(\cv)\in \vY_k \quad\text{and}\quad 
\sum_{\cw\in \cJ} \cp_l(\cw)\in \vY_l
\]
the following conditions hold
\begin{equation}\label{eq:subalgebraY_kl}
\left[y, \sum_{\cw\in \cJ} \cp_l(\cw)\right]\in \vY_{l+1}\quad \text{and}\quad 
\left[ \sum_{\cv\in \cI} \cp_k(\cv),\sum_{\cw\in \cJ} \cp_l(\cw) \right]\in \vY_{k+l}
\end{equation}
where $y\in \vY_1$. Notice that the nilpotence of $\vM$ gives $\vY_{k+l}=\{0\}$ whenever $k+l>\nu$. The first condition of \eqref{eq:subalgebraY_kl} is straightforward:
\[
\left[y, \sum_{\cv\in \cJ} \cp_l(\cw)\right]=
-\sum_{\cv\in \cJ} \cp_{l+1}((\cw,y))\in \vY_{l+1}.
\]
The second one for $k=2$ is a consequence of the Jacobi identity, that yields
\[
\begin{split}
\left[[v_1,v_2], \sum_{\cw\in \cJ} \cp_l(\cw) \right]&=
-\left[\left[v_2,\sum_{\cw\in \cJ} \cp_l(\cw)  \right], v_1\right]-\left[\left[\sum_{\cw\in \cJ} \cp_l(\cw),v_1  \right], v_2\right] \\
&=\sum_{\cw\in \cJ} \cp_{l+2}((\cw,v_2,v_1)) -
\sum_{\cw\in \cJ} \cp_{l+2}((\cw,v_1,v_2))\in\vY_{l+2}
\end{split}
\]
for every positive integer $l$. 
If we assume the second condition of \eqref{eq:subalgebraY_kl} to hold for a
fixed $k\ge2$ and every $l\ge1$, then setting 
\[
\cv=(v_1,v_2,\ldots,v_{k+1})\in\vY_1^{k+1}\quad\text{and}\quad 
\tilde \cv=(v_1,v_2,\ldots,v_k)\in\vY_1^k,
\]
we consider the product
\[
\begin{split}
\left[\cp_{k+1}(\cv), \sum_{\cw\in \cJ} \cp_l(\cw) \right]&=
\left[[\cp_k(\tilde\cv),v_{k+1}], \sum_{\cw\in \cJ} \cp_l(\cw) \right] \\
&=-\left[\left[v_{k+1}, \sum_{\cw\in \cJ} \cp_l(\cw)\right],\cp_k(\tilde\cv) \right]
-\left[\left[\sum_{\cw\in \cJ} \cp_l(\cw),\cp_k(\tilde\cv)\right],v_{k+1} \right]\\
&=-\left[\cp_k(\tilde\cv), \sum_{\cw\in \cJ} \cp_{l+1}((\cw,v_{k+1})) \right]
+\sum_{\cw\in \cJ}\left[[\cp_k(\tilde\cv), \cp_l(\cw)],v_{k+1} \right].
\end{split}
\]
All addends of the previous sum belong to $\vY_{k+l+1}$ by the inductive assumption, hence establishing conditions \eqref{eq:subalgebraY_kl}. The linear subspace
\[
V=\vY_1+\vY_2+\cdots +\vY_\nu\subset\vM
\]
is closed under dilations, since so are all the subspaces $\vY_j$.
Conditions \eqref{eq:subalgebraY_kl} immediately show that $V$ is a
Lie subalgebra of $\vM$. It turns out that its closure
\[
\vV=\overline{V}
\]
is a Banach Lie subalgebra, that is also closed under dilations, 
due to the continuity of the Lie product and of dilations. Then the same argument
in the proof of \cite[Proposition~7.2]{Mag} shows that $\vV$ is actually 
a direct sum of subspaces $U_1,\ldots,U_\nu$ with 
$[U_i,U_j]\subset U_{i+j}$ and $U_{i+j}=\set{0}$ whenever $i+j>\nu$.
Since the group operation is given by the Baker-Campbell-Hausdorff formula, 
it turns out that $\vV$  is a Banach homogeneous subgroup of $\vM$.

To show that $\vV$ is also separable, we first consider the countable subsets
\[
\lan\lan\cN\ran\ran_k=\left\{\sum_{v\in \cI}  \cp_k(v):\  \cI\subset\big(\lan\lan\cN\ran\ran\big)^k \ \text{is finite} \right\}
\]
where $2\le k\le\nu$ and define the closed subset
\[
\vW=\overline{\lan\lan\cN\ran\ran+\lan\lan\cN\ran\ran_2+\cdots 
+\lan\lan\cN\ran\ran_\nu}.
\]
Clearly $\vW$ is separable, being all the addends $\lan\lan \cN\ran\ran$ and  $\lan\lan \cN\ran\ran_k$ countable.
The simple inclusion $\lan\lan\cN\ran\ran_k \subset\vY_k$
for $k\ge2$ joined with $\vY_1=\overline{\lan\lan\cN\ran\ran}$ immediately lead us to the following
\[
\lan\lan\cN\ran\ran+\lan\lan\cN\ran\ran_2+\cdots 
+\lan\lan\cN_\nu\ran\ran_\nu\subset V,
\]
that gives $\vW\subset\vV$.

The opposite inclusion follows from the following
\begin{equation}\label{eq:inclusionV}
\vW\supset V.
\end{equation}
To prove this, we consider any element 
\[
w=v+\sum_{k=2}^\nu \sum_{\cv \in \cI_k} \cp_k(\cv)\in V,
\]
with $v\in\vY_1$ and $\cI_k$ finite subset of $\vY_1^k$.
Since $\overline{\lan\lan \cN\ran\ran}=\vY_1$, we find a sequence $\{y_n\}$ of $\lan\lan\cN\ran\ran$ with $y_n\to v$. For $k\ge2$ and any $\cv\in \cI_k$ we find a sequence $\{\cv_n\}\subset(\lan\lan\cN\ran\ran)^k$ such that $\cv_n\to \cv$. By the continuity of $\cp$, we get
\[
\cp_k(\cv_n)\to\cp_k(\cv) \quad \text{ as $n\to\infty$}.
\]
We conclude that 
\[
z_n=y_n+\sum_{k=2}^\nu \sum_{\cv \in \cI_k} \cp_k(\cv_n)\in
\lan\lan\cN\ran\ran+\lan\lan\cN\ran\ran_2+\cdots 
+\lan\lan\cN\ran\ran_\nu
\]
and $z_n\to w\in\vW$, showing the validity of \eqref{eq:inclusionV}
and then concluding the proof.
\end{proof}

\begin{Teo}\label{Teo:separable_image}
If $A\subset\vG$ and $f:A\to \vM$ is continuous, then there exists a separable Banach homogeneous
group $\vM_0\subset\vM$ such that $f(A)\subset \vM_0$.
\end{Teo}
\begin{proof}
Let $\cN\subset \vM$ be a countable subset such that $f(A)\subset\overline\cN$.
Then Theorem~\ref{teo:sepsub} provides us with a separable Banach homogeneous group
$\vM_0\subset\vM$ containing $\overline\cN$, therefore concluding the proof.
\end{proof}

\subsection{Proof of Lemma~\ref{estMagn}}\label{Sect:estMagn}
Fix $b>0$. Since each map $\pi_i$ is continuous and linear, the following inequality holds for $|x|\leq b$, where $c>0$ only depends on $\vM$:
\begin{align}\label{stima}
|\pi_i(x)|^{1/i} \le c^{1/i}|x|^{1/i}&=c^{1/i}(1+b)^{1/i}\left(\frac{|x|}{1+b}\right)^{1/i}\\
& \le c^{1/i}(1+b)^{\frac1i-\frac{1}{\nu}}|x|^{1/\nu} \nonumber \\
&\le \max\{c,\ 1\}\,(1+b)^{1-\frac{1}{\nu}}\, |x|^{1/\nu}, \nonumber
\end{align}
By \eqref{Banhom} and \eqref{stima} there exists $C_{1,b}>0$, depending on $b$, such that 
\begin{equation}\label{eq:estimhomnormBan}
\| x\|\le C_{1,b}\, |x|^{1/\nu} \quad\text{for}\quad |x|\le b.
\end{equation}
Conversely, from \eqref{Banhom} one may easily see that there exists $C_{2,b}>0$ such that
\begin{align}\label{d}
|x|\le  C_{2,b} \quad\mbox{for}\quad \|x\|\le b.
\end{align}
The same formula also yields
\begin{equation}\label{eq:revestimhomnormBan}
|x| \le \sum_{i=1}^{\nu} |\pi_{i}(x)|\le \max_{1\le i\le \nu}\set{\frac1{(\sigma_i)^i}}\sum_{i=1}^{\nu} \pa{C_{1,b}\, b^{1/\nu}}^{i-1} \|x\|= C_{3,b}\,\|x\| \qquad \mbox{for $|x|\le b$}.
\end{equation}
We observe that for the addends in \eqref{dynkin} to be non-zero, the iterated nonassociative product must always start from the factor $x\circ y$ or $y\circ x$. Thus, the continuity of the Lie bracket yields
\begin{equation}\label{eq:ext[xy]}
|\mbox{\small $\underbrace{x\circ\cdots\circ x}_{p_1\; times}\circ\overbrace{y\circ\cdots\circ y}^{q_1\; times}\circ \cdots\circ  \underbrace{x\circ\cdots\circ x}_{p_k\; times}\circ\overbrace{y\circ\cdots\circ y}^{q_k\; times }$ }|\le C^{m-2} b^{m-2} |[x,y]|
\end{equation}
where $\sum_{i=1}^{k} (p_{i}+q_{i})=m$ and we have assumed in addition that $|x|,|y|\le b$ for some $b>0$. 
\begin{Claim}\label{Claim 2}
Let $b>0$. Then there exists $C_{4,b}>0$, depending on $b$, such that 
\begin{eqnarray}\label{eq:conjugate_est}
\|y^{-1}xy\|\leq C_{4,b}\;|x|^{1/\nu} \qquad \mbox{for }|x|,|y|\leq b.
\end{eqnarray}
\end{Claim}

\begin{proof}
Recalling $y^{-1}=-y$ and applying Dynkin's formula with the pairs $-y$, $xy$ then $x, y$, we obtain
\begin{eqnarray}\label{bchassoc}
y^{-1}xy=x+\sum_{m=2}^\nu P_m(x,y)+\sum_{m=2}^\nu P_m(-y,x y).
\end{eqnarray}
Taking into account \eqref{eq:ext[xy]}, it is easy to see that there exists $C_{5,b}>0$, such that 
\[
|y^{-1}xy|\le C_{5,b}|x|+
\left|\sum_{m=2}^\nu P_m(-y,x y)\right|.
\]
Another application of \eqref{eq:ext[xy]}, joined with \eqref{dynkin}, gives a constant $C_{6,b}>0$ such that 
\[
| x y|\leq C_{6,b} \qquad \mbox{for }|x|,|y|\le b.
\]
Again using \eqref{eq:ext[xy]}, it follows that
\begin{eqnarray}\label{iterest}
|P_m(-y,xy)| \leq C_{7,b}  \, |[y,xy]|
\end{eqnarray}
for each $m=2,\ldots,\nu$ and a suitable constant $C_{7,b}>0$, depending on $b$.
Thus, we have proved the existence of $C_{8,b}>0$, depending on $b$, such that \begin{equation}\label{eq:y-1xyest}
|y^{-1}xy|\le C_{8,b}\,(|x|+|[y,xy]|).
\end{equation}
Now, we observe that
\[
[y,xy]=\big[y, x+\sum_{m=2}^\nu P_m(x,y)\big]=[y,x]+\big[y,\sum_{m=2}^\nu P_m(x,y)\big].
\]
This together with the continuity of $\left[\cdot,\cdot\right]$ leads us 
to the estimate
\[
|[y,x y]|\leq C_0\,\Big( |x||y|+ |y| \sum_{m=2}^\nu |P_m(x,y)|\Big)
\]
where $C_0>0$ only depends on the Banach norm chosen on $\vM$.
As a consequence, there exists $C_{9,b}>0$ such that
\begin{equation}\label{eq:|[y,xy]|}
|[y,x y]|\leq C_{9,b} |x|\quad\text{for} \quad |x|,|y|\le b.
\end{equation}
The previous inequality joint with \eqref{eq:y-1xyest} gives
$|y^{-1}xy| \leq C_{10,b}$ for some $C_{10,b}>0$ depending on $b$.
This bound on $|y^{-1}xy|$ allows us to apply \eqref{eq:estimhomnormBan}, 
that joined with \eqref{eq:y-1xyest} and \eqref{eq:|[y,xy]|} gives \eqref{eq:conjugate_est}.
\end{proof}

\begin{Claim}\label{Claim 3}
Let $N$ be a positive integer and let $A_j,B_j\in\vM$ with $j=1,\ldots,N$. Let $b>0$ be such that $|B_jB_{j+1}\cdots B_N|\leq b$ and $|B_j^{-1}A_j|\leq b$
for every $j=1,\ldots,N$. Then there exists $C>0$ such that
\begin{eqnarray}\label{eq:estimprod}
\rho\big(A_1A_2\cdots A_N,B_1B_2\cdots B_N\big)\leq   \rho (A_N,B_N) +C\sum_{j=1}^{N-1}|A_j^{-1} B_j|^{1/\nu}.
\end{eqnarray}
\end{Claim}
\begin{proof}
Define $\hat{B}_j=B_jB_{j+1}\cdots B_N$ and $\hat{A}_j=A_jA_{j+1}\cdots A_N$. 
Using the left invariance of $\rho$ and Claim~\ref{Claim 2}, we obtain
\begin{align*}
\rho\big(\hat{A}_1,\hat{B}_1\big)=\rho\big(\hat{A}_2,A_1^{-1}B_1\hat{B}_2\big)&\leq  \rho(\hat A_2,\hat B_2)+\|\hat B_2^{-1}A_{1}^{-1}B_1\hat B_2\| \\
&\leq  \rho(\hat A_2,\hat B_2)+C_{4,b} |A_1^{-1}B_1|^{1/\nu} \\
&\leq  \rho(\hat A_3,\hat B_3)+C_{4,b} |A_2^{-1}B_2|^{1/\nu}+C_{4,b} |A_1^{-1}B_1|^{1/\nu}\\
&\le  \rho(A_N,B_N) +C \sum_{j=1}^{N-1}|A_j^{-1} B_j|^{1/\nu}.\qedhere
\end{align*}
\end{proof}
From the hypotheses of Lemma~\ref{estMagn}, taking into account \eqref{d}, we see
\[
|B_jB_{j+1}\cdots B_N|\le C_{2,b} \quad\mbox{and}\quad  |A_j^{-1}B_j|\leq  C_{2,b}\qquad \mbox{for }j=1,\ldots,N.
\]
Hence we may apply Claim~\ref{Claim 3}. Recalling that $\rho(A_N,B_N)\leq b$ and using \eqref{eq:revestimhomnormBan} we obtain
\[
\rho\big(A_1A_2\cdots A_N,B_1B_2\cdots B_N\big)\leq  \rho(A_N,B_N) +C_{3,C_{2,b}}\sum_{j=1}^{N-1} \|A_j^{-1} B_j\|^{1/\nu}
\leq  C_b \sum_{j=1}^N\rho(A_j, B_j)^{1/\nu},
\]
where $C_b=\max\set{b^{1-\frac1\nu},C_{3,C_{2,b}}}$. This concludes the proof of Lemma~\ref{estMagn}.

\subsection{Measurability of points where a directional derivative exists}\label{Sect:Measurability}

The aim of this section is to prove that the set $D_{f,\zeta}$ in the proof of 
Theorem~\ref{lineardensityderivatives} is measurable. Recall that 
$f:A\to\vM$ is Lipschitz and $A\subset\vG$ is measurable.
It was shown earlier that the set $A_{\zeta}$ of points at which $A$ is dense in direction $\zeta$ is measurable and $\mu(A\setminus A_{\zeta})=0$. The set $D_{f,\zeta}$ is the subset of points in $A_\zeta$ where $f$ is differentiable in direction $\zeta$.
By Theorem~\ref{Teo:separable_image}, we can assume that the target $\vM$ is a separable Banach homogeneous group.
Fix $z_{i}\in \vM$ with $\overline{\set{z_i: i\in\N}}=\vM$. Let $P_i(z)=\rho(z,z_i)$ for $z\in\vM$. For $y\in A$ and $t>0$, define measurable functions
$g^i_{t,y,+},g^i_{t,y,-}:A_\zeta\to\bbR\cup\set{+\infty,-\infty}$ by: 
\[
g^i_{t,y,\pm}(x)=\left\{\begin{array}{ll} 
P_i\pa{\delta_{1/t}\pa{f(x)^{-1}f(y)}} &  \text{if $y\in A\cap B(x\delta_t\zeta,t^2)$}  \\
\pm\infty & \text{if $y\in A\setminus B(x\delta_t\zeta,t^2)$.}
\end{array}\right.
\]
Fix a countable set $\cS\subset A$ with $A\subset \overline\cS$. Define
\[
x\to\ph(x)=\lim_{R\to+\infty}\sup_{i\in\N}\, \sup_{\substack{0<s,t<1/R \\ s,t\in\Q}}
\,\sup_{(y,z)\in\cS^2}\pa{g^i_{t,y,-}(x)-g^i_{s,z,+}(x)},
\]
which is well defined and measurable on $A_\zeta$. We claim that $f$ is differentiable at $x$ in direction $\zeta$ if and only if $\ph(x)=0$.

First suppose $x\in A_\zeta$ and $f$ is differentiable at $x$ in direction $\zeta$. Proposition~\ref{Pro:fDiffDirZeta} implies that
\[
\lim_{\;\,t\downarrow 0,\, x\delta_t\zeta\in A} \delta_{1/t}\pa{f(x)^{-1}f(x\delta_t\zeta)}
=\der^+f(x,\zeta).
\]
Since $x$ is a density point of $A$ in direction $\zeta$, there exists $T_t>0$ with $x\delta_{T_t}\zeta\in A$ and $T_t/t\to1$ as $t\downarrow0$. For $y\in B(x\delta_t\zeta,t^2)$, let
\[
E_{x,y,t}=\delta_{1/t}\pa{f(x\delta_{T_t}\zeta)^{-1}f(y)}
\]
Taking into account that $\zeta$ is horizontal, we have
\[
\rho(E_{x,y,t})\le L\, \frac{d(x\delta_{T_t}\zeta,y)}t< \frac Lt
\pa{|T_t-t|\,d(\zeta)+t^2}=L\theta_t,
\]
where $\theta_t\to0$ as $t\downarrow0$. We set
\[
(\delta f)_{x,\zeta,t}=\delta_{1/t}\pa{f(x)^{-1}f(x\delta_{T_t}\zeta)},
\]
observing that $(\delta f)_{x,\zeta,t}\to\der^+f(x,\zeta)$ as $t\downarrow0$.
The following estimates hold
\[
\pal{P_i\pa{\delta_{1/t}\pa{f(x)^{-1}f(y)}}-P_i\pa{(\delta f)_{x,\zeta,t}}}=
\pal{P_i\pa{(\delta f)_{x,\zeta,t}E_{x,y,t} }-P_i\pa{(\delta f)_{x,\zeta,t}}} \le\rho(E_{x,y,t})<L \theta_t
\]
uniformly in $y\in B(x\delta_t\zeta,t^2)$ and $i\in\N$.
Since $P:\vM\to\ell^\infty$, $z\to (P_i(z))_{i\in\N}$ is an isometric embedding, we get
\[
\sup_{i\in\N}\pal{P_i\pa{(\delta f)_{x,\zeta,t}}-P_i\pa{\der^+f(x,\zeta)}}\to0\quad\text{as $t\downarrow0$.}
\]
If we fix $\ep>0$, then for some $R_\ep>0$
we obtain
\[
\pal{P_i\pa{\delta_{1/t}\pa{f(x)^{-1}f(y)}}-P_i\pa{\der^+f(x,\zeta)}}<\varepsilon/2
\] 
for $0<t<1/R_\ep$, $i\in\N$ and $y\in B(x\delta_t\zeta,t^2)$. In particular,
for $y\in B(x\delta_t\zeta,t^2)$ and $z\in B(x\delta_s\zeta,s^2)$, it follows that
\[
\pal{P_i\pa{\delta_{1/t}\pa{f(x)^{-1}f(y)}}-P_i\pa{\delta_{1/s}\pa{f(x)^{-1}f(z)}}}<\ep
\]
for $0<t,s<1/R_\ep$ and $i\in \mathbb{N}$. The fact that $A$ is dense at $x$ in direction $\zeta$ implies that, up to taking a larger $R_\ep$, the following
\[
\sup_{\substack{0<s,t<1/R_\ep\\ s,t\in\Q}}\,\sup_{(y,z)\in\cS^2}\pa{g^i_{t,y,-}(x)-g^i_{s,z,+}(x)}
\]
exactly equals 
\[
\sup_{\substack{0<s,t<1/R_\ep\\ s,t\in\Q}}\;\sup_{\substack{(y,z)\in\cS^2\\ y\in B(x\delta_t\zeta,t^2),\,z\in B(x\delta_s\zeta,s^2)}}P_i\pa{\delta_{1/t}\pa{f(x)^{-1}f(y)}}-P_i\pa{\delta_{1/s}\pa{f(x)^{-1}f(z)}}<\ep
\]
for every $i\in\N$. This proves that $\ph(x)=0$.

Conversely, we assume that $\ph(x)=0$ and wish to prove that $f$ is differentiable
at $x$ in direction $\zeta$. Let $\ep>0$ be arbitrarily fixed and choose $R_\ep>0$ such that
\begin{equation}\label{eq:gtpm}
\sup_{i\in\N}\sup_{\substack{0<s,t<1/R_\ep\\ s,t\in\Q}}\sup_{(y,z)\in\cS^2}\pa{g^i_{t,y,-}(x)-g^i_{s,z,+}(x)}<\ep.
\end{equation}
Observing that in the expression \eqref{eq:gtpm} we can exchange $(t,y)$ with $(s,z)$,
if $y\in B(x\delta_t\zeta,t^2)$ and $z\in B(x\delta_s\zeta,s^2)$ we get
\[
\pal{P_i\pa{\delta_{1/t}\pa{f(x)^{-1}f(y)}}-P_i\pa{\delta_{1/s}\pa{f(x)^{-1}f(z)}}}\le\ep
\]
for $0<s,t\le 1/R_\ep$. Since $x\in A$ is a density point in direction $\zeta$, we may consider again
$T_t>0$ such that $x\delta_{T_t}\zeta\in A$ and $T_t/t\to1$ as $t\downarrow0$.
Fix $\delta_0>0$ such that $T_t<2t$  for $0<t\le\delta_0$.
Whenever $0<s,t<1/2R_\ep$, we have $0<T_s,T_t<1/R_\ep$ and then
\begin{equation}\label{eq:P_il_i}
\pal{P_i\pa{\delta_{1/T_t}\pa{f(x)^{-1}f(x\delta_{T_t}\zeta)}}
-P_i\pa{\delta_{1/T_s}\pa{f(x)^{-1}f(x\delta_{T_s}\zeta)}}}\le\ep
\end{equation}
for all $i\in\N$. In particular, there exists
\[
l_i=\lim_{t\downarrow0} P_i\pa{\delta_{1/T_t}\pa{f(x)^{-1}f(x\delta_{T_t}\zeta)}}
\]
and we can pass to the limit in \eqref{eq:P_il_i} with respect to $s\downarrow0$,
that yields
\[
\pal{P_i\pa{\delta_{1/T_t}\pa{f(x)^{-1}f(x\delta_{T_t}\zeta)}}
-l_i}\le\ep.
\]
In particular, $w=(l_i)\in\ell^\infty$ and
\[
\|P\pa{\delta_{1/T_t}\pa{f(x)^{-1}f(x\delta_{T_t}\zeta)}}-w\|_{\ell^\infty}\to0\quad \text{as $t\downarrow0$}.
\]
Hence the following limit exists
\[
\lim_{t\downarrow0}\delta_{1/T_t}\pa{f(x)^{-1}f(x\delta_{T_t}\zeta)}.
\]
By Remark~\ref{derivindependent}, this shows that $f$ is differentiable at $x$ in direction $\zeta$.

We conclude that $D_{f,\zeta}=\set{x\in A_\zeta: \ph(x)=0}$, from which the measurability of $\ph$ gives our claim.



\begin{thebibliography}{99}
\bibitem{ABB} \texttt{Agrachev, A., Barilari, D., Boscain, U.}: \emph{Introduction to Riemannian and sub-Riemannian geometry (from hamiltonian viewpoint)}, notes available at http://webusers.imj-prg.fr/~davide.barilari/.
\bibitem{Alb} \texttt{Alberti, G., Marchese, A.}: \emph{On the differentiability of Lipschitz functions with respect to measures in the Euclidean space}, Geom. Funct. Anal. 26 (2016), no. 1, 1--66. 
\bibitem{AmbKir00} \texttt{Ambrosio, L., Kirchheim, B.}: \emph{Rectifiable sets in metric and Banach spaces}, Math. Ann. 318 (2000), no. 3, 527--555.
\bibitem{Bate15}\texttt{Bate, D.}: \emph{Structure of measures in Lipschitz differentiability spaces}, J. Amer. Math. Soc. 28 (2015), no. 2, 421--482. 
\bibitem{B94} \texttt{Bellaiche, A.}: \emph{The tangent space in sub-Riemannian geometry}, Journal of Mathematical Sciences, 83(4) (1994), 461--476.
\bibitem{BLU} \texttt{Bonfiglioli, A., Lanconelli, E., Uguzzoni, F.}: \emph{Stratified Lie groups and potential theory for their sub-laplacians}, Springer Monographs in Mathematics 26, New York, Springer-Verlag (2007).
\bibitem{Che99} Cheeger, J.: \emph{Differentiability of Lipschitz functions on metric measure spaces}, Geom. Funct. Anal. 9(3) (1999), 428--517.
\bibitem{CK} \texttt{Cheeger, J.,  Kleiner, B.}: \emph{Differentiability of Lipschitz maps from metric measure spaces to Banach spaces with the Radon-Nikodym property}. Geom. Funct. Anal. 19 (2009), no. 4, 1017--1028.
\bibitem{CMPSC1} \texttt{Citti, G., Manfredini, M., Pinamonti, A., Serra Cassano, F.}: \emph{Smooth approximation for intrinsic Lipschitz functions in the Heisenberg group}, Calc. Var. Partial Differential Equations 49(3--4) (2014), 1279--1308.
\bibitem{CMPSC2} \texttt{Citti, G., Manfredini, M., Pinamonti, A., Serra Cassano, F.}: \emph{Poincar\'e-type inequality for Lipschitz continuous vector fields}, J. Math. Pures Appl. (9) 105(3) (2016), 265--292.
\bibitem{CDPT07} \texttt{Capogna, L., Danielli, D., Pauls, S., Tyson, J.}: \emph{An introduction to the Heisenberg group and the sub-Riemannian isoperimetric problem}, Birkhauser, Progress in Mathematics, 259 (2007).
\bibitem{Fed69}\texttt{Federer, H.}: \emph{Geometric measure theory}, Die Grundlehren der mathematischen Wissenschaften, Springer-Verlag, New York (1969)
\bibitem{FS82} \texttt{Folland, G.B., Stein, E.}: \emph{Hardy spaces on homogeneous groups}, Princeton University Press (1982).
\bibitem{Fol75} \texttt{Folland, G.B.}: \emph{Subelliptic estimates and function spaces on nilpotent Lie groups}, Ark. Mat. 13(1--2) (1975), 161–-207.
\bibitem{FS} \texttt{Franchi, B., Serapioni, R.}: \emph{Intrinsic Lipschitz graphs within Carnot groups}, J. Geom. Anal.,  26(3) (2016), 1946--1994.
\bibitem{FSC2} \texttt{Franchi, B., Serapioni, R., Serra Cassano, F.}: \emph{Differentiability of intrinsic Lipschitz functions within Heisenberg groups}, J. Geom. Anal. 21(4) (2011), 1044--1084.
\bibitem{Gro96} \texttt{Gromov, M.}: \emph{Carnot-Carath\'eodory spaces seen from within}, Progress in Mathematics, 144 (1996), 79--323.
\bibitem{HajMal15} \texttt{Hajlasz, P., Malekzadeh, S.}: \emph{On Conditions for Unrectifiability of a Metric Space}, Anal. Geom. Metr. Spaces 3(1) (2015), 1--14.  
\bibitem{Hei01} \texttt{Heinonen, J.}: \emph{Lectures on analysis on metric spaces}, Universitext. Springer-Verlag, New York (2001).
\bibitem{LPS17} \texttt{Le Donne, E., Pinamonti, A., Speight, G.}: \emph{Universal differentiability sets and maximal directional derivatives in Carnot groups}, preprint available at arXiv:1705.05871.
\bibitem{LP03} \texttt{Lindenstrauss, J., Preiss, D.}: \emph{On Fr\'{e}chet differentiability of Lipschitz maps between Banach spaces}, Annals of Mathematics 157 (2003), 257--288.
\bibitem{LPT12} \texttt{Lindenstrauss, J., Preiss, D., Tiser, J.}: \emph{Fr\'{e}chet differentiability of Lipschitz functions and porous sets in Banach spaces}, Annals of Mathematics Studies 179, Princeton University Press (2012).
\bibitem{Mag01} \texttt{Magnani, V.}: \emph{Differentiability and area formula on stratified Lie groups}, Houston J. Math. 27(2) (2001), 297--323.
\bibitem{Mag04} \texttt{Magnani, V.}: \emph{Unrectifiability and rigidity in stratified groups}, Archiv der Mathematik 83(6) (2004), 568-–576.
\bibitem{Mag} \texttt{Magnani, V.}: \emph{Towards differential calculus in stratified groups}, J. Aust. Math. Soc. 95(1) (2013), 76--128.
\bibitem{MagTap14} \texttt{Magnani, V., Rajala, T.}: \emph{Radon-Nikodym property and area formula for Banach homogeneous group targets}, Int. Math. Res. Not. IMRN (2014), no. 23, 6399--6430.
\bibitem{Mon02} \texttt{Montgomery, R.}: \emph{A tour of subriemannian geometries, their geodesics and applications}, American Mathematical Society, Mathematical Surveys and Monographs, 91 (2006).
\bibitem{PP} \texttt{Pansu, P.}: \emph{Metriques de Carnot-Carath\'{e}odory et quasiisometries des espaces symetriques de rang un}, Annals of Mathematics 129(1) (1989), 1--60.
\bibitem{PinS15} \texttt{Pinamonti, A., Speight, G.}: \emph{A measure zero universal differentiability set in the Heisenberg group}, Math. Ann., 368(1) (2017), 233--278.
\bibitem{PinS16} \texttt{Pinamonti, A., Speight, G.}: \emph{Porosity, differentiability and Pansu's theorem}, to appear in J. Geom. Anal. DOI:10.1007/s12220-016-9751-6
\bibitem{PS3} \texttt{Pinamonti, A., Speight, G.}: \emph{Structure of Porous Sets in Carnot Groups}, preprint available at arXiv:1607.04681.
\bibitem{PS4} \texttt{Pinamonti, A., Speight, G.}: \emph{A Measure Zero UDS in the Heisenberg Group}, to appear in Bruno Pini Mathematical Analysis Seminar. 
\bibitem{Pre90} \texttt{Preiss, D.}: \emph{Differentiability of Lipschitz functions on Banach spaces}, J. Funct. Anal. 91(2) (1990), 312--345.
\bibitem{PS15} \texttt{Preiss, D., Speight, G.}: \emph{Differentiability of Lipschitz functions in Lebesgue null sets}, Inventiones Mathematicae 199(2) (2015), 517--559.
\bibitem{PZ01} \texttt{Preiss, D., Zajicek, L.}: \emph{Directional derivatives of Lipschitz functions}, Israel Journal of Mathematics 125 (2001), 1--27.
\bibitem{Rog} \texttt{Rogovin, K.}: \emph{Non-smooth analysis in infinite dimensional Banach homogeneous groups}, J. Convex Anal. 14(4) (2007), 667--697.
\bibitem{Sem96} \texttt{Semmes, S.}: \emph{On the nonexistence of bi-Lipschitz parameterizations and geometric problems about $A_{\infty}$-weights}, Revista Matematica Iberoamericana, 12(2) (1996), 337--410.
\bibitem{SC16} \texttt{Serra Cassano, F.}: \emph{Some topics of geometric measure theory in Carnot groups}, Volume I, EMS Series of Lectures in Mathematics. European Mathematical Society (EMS), Z\"urich, to appear.
\bibitem{Zaj87} \texttt{Zajicek, L.}: \emph{Porosity and $\sigma$-porosity}, Real Analysis Exchange 13(2) (1987/1988), 314--350.
\bibitem{Zaj05} \texttt{Zajicek, L.}: \emph{On $\sigma$-porous sets in abstract spaces}, Abstract and Applied Analysis 2005(5), 509--534.
\end{thebibliography}
\end{document}